\documentclass[12pt, oneside]{article}	
\usepackage{geometry}                		
\usepackage{graphicx}				
\usepackage{amssymb, amsmath, amsthm, proof, lscape, tikz, tcolorbox, setspace, changepage, hyperref, graphicx, enumerate, accents, framed}
\theoremstyle{definition}
\newtheorem{theorem}{Theorem}
\newtheorem{lemma}{Lemma}
\newtheorem{proposition}{Proposition}
\newtheorem{corollary}{Corollary}
\newtheorem{subclaim}{Subclaim}
\usetikzlibrary{matrix}
\tikzset{
modal/.style={>=stealth,shorten >=1pt,shorten <=1pt,auto,node distance=1.5cm,
semithick},
world/.style={circle,draw,minimum size=0.5cm,fill=gray!15},
point/.style={circle,draw,inner sep=0.5mm,fill=black},
reflexive above/.style={->,loop,looseness=20,in=120,out=60},
reflexive below/.style={->,loop,looseness=20,in=240,out=300},
reflexive left/.style={->,loop,looseness=20,in=150,out=210},
reflexive right/.style={->,loop,looseness=20,in=30,out=330}
}

\newenvironment{subproof}[1][\proofname]{%
  \begin{proof}[#1]%
}{%
  \end{proof}%
}

\title{Binary Kripke Semantics for a Strong Logic for Naive Truth} 
\author{Ben Middleton \\ University of Notre Dame}
\date{}	

\begin{document}
\maketitle

\begin{abstract}
I show that the logic $\textsf{TJK}^{d+}$, one of the strongest logics currently known to support the naive theory of truth, is obtained from the Kripke semantics for constant domain intuitionistic logic by (i) dropping the requirement that the accessibility relation is reflexive and (ii) only allowing reflexive worlds to serve as counterexamples to logical consequence. In addition, I provide a simplified natural deduction system for $\textsf{TJK}^{d+}$, in which a restricted form of conditional proof is used to establish conditionals.

\end{abstract}

\section{Introduction}

Let $\mathcal{L}_T$ be the first-order language with primitive operators $\top$, $\bot$, $\wedge$, $\vee$, $\rightarrow$, $\forall$, $\exists$ whose signature consists of a binary relation symbol $=$ for identity, a unary relation symbol $T$ for truth, a constant symbol $0$ for the number zero and an $n$-ary function symbol $f_e$ for the $n$-ary primitive recursive function with index $e$. We identify a linguistic object with its G{\"o}del code. In the absence of a primitive negation operator, the naive theory of truth (\textsf{NT}) can be identified with the following set of $\mathcal{L}_T$-sentences:
\begin{align*}
\text{A1. } &\forall x \hspace{1mm} x = x \\
\text{A2. } &\forall x \forall y(x = y \rightarrow y = x) \\
\text{A3. } &\forall x \forall y \forall z(x = y \wedge y = z \rightarrow x = z) \\
\text{I$e$. } &\forall \overline{x} \forall \overline{y}(\textstyle\bigwedge_i x_i = y_i \rightarrow f_e(\overline{x}) = f_e(\overline{y})) \\
\text{A4. } &\forall x \forall y(x = y \wedge T(x) \rightarrow T(y)) \\
\text{A5. } &\forall x(s(x) = 0 \rightarrow \bot) \\
\text{A6. } &\forall x \forall y(s(x) = s(y) \rightarrow x = y) \\
\text{D$e$. } &\text{the definition of } f_e \text{ for } f_e \neq s \\
\text{Ind. } &\forall \overline{x}[\forall y(\phi(y, \overline{x}) \rightarrow \phi(s(y), \overline{x})) \rightarrow (\phi(0, \overline{x}) \rightarrow \forall y \phi(y, \overline{x}))] \\
\text{TB. } &T\ulcorner \phi \urcorner \leftrightarrow \phi
\end{align*}
where $\ulcorner \phi \urcorner$ is the numeral for $\phi$ and $\phi \leftrightarrow \psi$ abbreviates $(\phi \rightarrow \psi) \wedge (\psi \rightarrow \phi)$.\begin{footnote}{Note, in particular, the form of the induction schema. The more usual formulation $\forall \overline{x}[\phi(0, \overline{x}) \wedge \forall y(\phi(y, \overline{x}) \rightarrow \phi(s(y), \overline{x})) \rightarrow \forall y \phi(y, \overline{x})]$ is false in the standard model for the closure of $\textsf{NT}$ under $\textsf{TJK}^{d+}$ (see the appendix of this paper).}\end{footnote} Closing $\textsf{NT}$ under classical first-order logic (\textsf{CQL}) results in the trivial theory. So if we want to accept \textsf{NT} without being committed to \textit{everything}, we need to weaken \textsf{CQL}. Even if we weaken \textsf{CQL} to the point where \textsf{NT} becomes non-trivial, \textsf{NT} might still behave undesirably in other respects. Most obviously, \textsf{NT} might be $\omega$-inconsistent, in the sense that either (i) $\textsf{NT}$ implies $\phi(\dot{n})$ for every $n$ but $\textsf{NT} \cup \{\forall v \phi\}$ explodes or (ii) $\textsf{NT}$ implies $\exists v \phi$ but $\textsf{NT} \cup \{\phi(\dot{n})\}$ explodes for every $n$. In this case, we cannot interpret the quantifiers in \textsf{NT} as restricted to $\omega$. Say that a subclassical logic \textit{supports} \textsf{NT} iff \textsf{NT} is $\omega$-consistent in the logic (and hence non-trivial). One of the strongest logics currently known to support \textsf{NT} is the logic $\textsf{TJK}^{d+}$, which is obtained from positive basic relevant logic with $\exists$-Elim ($\textsf{B}^{d+}$) by adding the following axioms for $\rightarrow$:
$$
(\phi \rightarrow \psi) \wedge (\psi \rightarrow \chi) \rightarrow (\phi \rightarrow \chi)
\qquad
\phi \rightarrow (\psi \rightarrow \phi)
$$
$$
(\phi \rightarrow \psi) \rightarrow ((\chi \rightarrow \phi) \rightarrow (\chi \rightarrow \psi))
\qquad
(\phi \rightarrow \psi) \rightarrow ((\psi \rightarrow \chi) \rightarrow (\phi \rightarrow \chi)).
$$
\noindent In this paper, I accomplish two tasks. First, I show that $\textsf{TJK}^{d+}$ can be given a simplified natural deduction system, in which a restricted form of conditional proof is used to establish conditionals. Second, I show that $\textsf{TJK}^{d+}$ is exactly the logic obtained from the Kripke semantics for constant domain intuitionistic logic by (i) dropping the requirement that the accessibility relation is reflexive and (ii) only allowing reflexive worlds to serve as counterexamples to logical consequence.

\section{The relevant hierarchy} 

In this section, I give an overview of the hierarchy of logics obtained from $\textsf{B}^{d+}$ by adding axioms for $\rightarrow$.\begin{footnote}{The definitions of $\textsf{B}^{d+}$, $\textsf{DJ}^{d+}$ and $\textsf{TJ}^{d+}$ are taken from Brady (1984).}\end{footnote} Let $\mathcal{L}$ be an arbitrary first-order language and let $\mathcal{L}^+ = \mathcal{L} \cup \{a_i: i \in \omega\}$, where each $a_i$ is a fresh constant symbol (the $a_i$ serve in proofs as names of arbitrarily chosen objects). $\textsf{B}^{d+}$ is axiomatized by the following natural deduction system over $\mathcal{L}^+$, where 
\begin{enumerate}[(C1)]
\item only sentences may occur in proofs,\begin{footnote}{This requirement determines which free variables, if any, the subformulas of an inference rule may contain (e.g.\ $\phi$ may not contain free variables in CD).}\end{footnote}
\item $a_i$ may not occur in $\phi$ or in any open assumption in the main subproof of $\forall$-Int,
\item $a_i$ may not occur in $\phi$, $\psi$ or in any open assumption besides $\phi(a_i)$ in the right main subproof of $\exists$-Elim (we refer to C2 and C3 as the \textit{eigenvariable constraints}),
\item all open occurrences of $\phi(a_i)$ in the right main subproof of $\exists$-Elim must be discharged (for the remaining inference rules, an arbitrary number of open occurrences --- including zero --- of the assumption in square brackets may be discharged from the relevant subproof).

\

\end{enumerate}
$$
[\top]\hspace{2mm}(\top\text{-Int})
\qquad
\infer[(\bot\text{-Elim})]
	{\phi}
	{\bot}
$$
$$
\infer[(\wedge\text{-Int})]
	{\phi \wedge \psi}
	{
	\phi
	&
	\psi
	}
\qquad
\infer[(\wedge\text{-Elim})]
	{\phi / \psi}
	{\phi \wedge \psi}
$$
$$
\infer[(\vee\text{-Int})]
	{\phi \vee \psi}
	{\phi / \psi}
\qquad
\infer[(\vee\text{-Elim})]
	{\chi}
	{
	\phi \vee \psi
	&
	\infer*
		{\chi}
		{[\phi]}
	&
	\infer*
		{\chi}
		{[\psi]}
	}
$$
$$
\infer[(\rightarrow\hspace{-1mm}\text{-Elim})]
	{\psi}
	{
	\phi
	&
	\phi \rightarrow \psi
	}
$$
\begin{framed}
$$
[\phi \rightarrow \phi]
$$
$$
[\phi \rightarrow \top]
\qquad
[\bot \rightarrow \phi]
$$
$$
[(\chi \rightarrow \phi) \wedge (\chi \rightarrow \psi) \rightarrow (\chi \rightarrow \phi \wedge \psi)]
\qquad
[\phi \wedge \psi \rightarrow \phi/\psi]
$$
$$
[\phi/\psi \rightarrow \phi \vee \psi]
\qquad
[(\phi \rightarrow \chi) \wedge (\psi \rightarrow \chi) \rightarrow (\phi \vee \psi \rightarrow \chi)] 
$$
$$
[\phi \wedge (\psi \vee \chi) \rightarrow (\phi \wedge \psi) \vee (\phi \wedge \chi)]
$$
$$
[\forall v(\phi \rightarrow \psi) \rightarrow (\phi \rightarrow \forall v \psi)]
\qquad
[\forall v \phi \rightarrow \phi(t)]
$$
$$
[\phi(t) \rightarrow \exists v \phi]
\qquad
[\forall v(\phi \rightarrow \psi) \rightarrow (\exists v \phi \rightarrow \psi)]
$$
$$
[\forall v(\phi \vee \psi) \rightarrow \phi \vee \forall v \psi]
\qquad
[\phi \wedge \exists v \psi \rightarrow \exists v(\phi \wedge \psi)]
$$
$$
\infer
	{(\psi \rightarrow \chi) \rightarrow (\phi \rightarrow \gamma)}
	{
	\phi \rightarrow \psi
	&
	\chi \rightarrow \gamma
	}
$$ 
\end{framed}
$$
\infer[(\forall\text{-Int})]
	{\forall v \phi}
	{\phi(a_i)}
\qquad
\infer[(\forall\text{-Elim})]
	{\phi(t)}
	{\forall v \phi}
\qquad
\infer[(\text{CD})]
	{\phi \vee \forall v \psi}
	{\forall v (\phi \vee \psi)}
$$
$$
\infer[(\exists\text{-Int})]
	{\exists v \phi}
	{\phi(t)}
\qquad
\infer[(\exists\text{-Elim})]
	{\psi}
	{
	\exists v \phi
	&
	\infer*
		{\psi}
		{[\phi(a_i)]}
	}
$$
Although this natural deduction system contains some redundancies, it is useful to view $\textsf{B}^{d+}$ as the result of deleting $\rightarrow$-Int from the natural deduction system for constant domain intuitionistic logic ($\textsf{IQL}_\textsf{CD}$) and replacing it with the boxed rules. Since $\textsf{NT}$ explodes in $\textsf{IQL}_\textsf{CD}$, the basic goal of naive truth theory, in the absence of a primitive negation operator, is to discover how close to the full strength of $\rightarrow$-Int we can get before reaching $\omega$-inconsistency. $\textsf{B}^{d+}$ is the positive fragment of the logic for naive truth theory endorsed by Beall (2009).

\subsection{$\textsf{DJ}^{d+}$}

We obtain the logic $\textsf{DJ}^{d+}$ by adding the transitivity axiom to the natural deduction system for $\textsf{B}^{d+}$:
$$
[(\phi \rightarrow \psi) \wedge (\psi \rightarrow \chi) \rightarrow (\phi \rightarrow \chi)].
$$
\noindent $\textsf{DJ}^{d+}$ is the positive fragment of the logic for naive truth theory endorsed by Brady (2006).

\subsection{$\textsf{TJ}^{d+}$}

We obtain the logic $\textsf{TJ}^{d+}$ by adding the suffixing and prefixing axioms to the natural deduction system for $\textsf{DJ}^{d+}$:
$$
[(\phi \rightarrow \psi) \rightarrow ((\psi \rightarrow \chi) \rightarrow (\phi \rightarrow \chi))]
\qquad
[(\phi \rightarrow \psi) \rightarrow ((\chi \rightarrow \phi) \rightarrow (\chi \rightarrow \psi))].
$$

\subsection{$\textsf{TJK}^{d+}$}

We obtain the logic $\textsf{TJK}^{d+}$ by adding the weakening axiom to the natural deduction system for $\textsf{TJ}^{d+}$:
$$
[\phi \rightarrow (\psi \rightarrow \phi)].
$$
\noindent The logic $\textsf{TJK}^+$ is obtained from the natural deduction system for $\textsf{TJK}^{d+}$ by deleting $\exists$-Elim. In fact, we show in \S4.2 that $\textsf{TJK}^{d+} = \textsf{TJK}^+$, so $\exists$-Elim is redundant in $\textsf{TJK}^{d+}$. It was shown by Bacon (2013a) that if $\textsf{NT}$ proves $\phi(\dot{n})$ in $\textsf{TJK}^+$ for every $n$ then $\textsf{NT} \cup \{\forall v \phi\}$ does not explode in $\textsf{TJK}^+$. This is weaker than full $\omega$-consistency (as defined in the introduction), since it remains possible that $\textsf{NT}$ proves $\exists v \phi$ even though $\textsf{NT} \cup \{\phi(\dot{n})\}$ explodes for every $n$. However, it was recently shown by Field, Lederman and {\O}gaard (2017) that $\textsf{TJK}^{d+}$ (and hence $\textsf{TJK}^+$) satisfies full $\omega$-consistency. This result was achieved by, in effect, building a standard model for the closure of $\textsf{NT}$ under $\textsf{TJK}^{d+}$. However, since Field-Lederman-{\O}gaard were working in the context of naive set theory, it is useful to see the standard model constructed explicitly for $\textsf{NT}$. I have therefore included the explicit construction in the appendix (and also give the construction in the framework of the binary Kripke semantics defined in the next section).

\section{Binary Kripke semantics for $\textsf{TJK}^{d+}$}

In this section, I introduce the binary Kripke semantics for $\textsf{TJK}^{d+}$. An $\mathcal{L}$-model is a $4$-tuple $\mathfrak{M} = \langle W, \prec, M, |\mathord{\cdot}| \rangle$ such that $W$ is a non-empty set (the set of worlds), $\prec$ is a transitive binary relation on $W$ (the accessibility relation), $M$ is a non-empty set (the domain of quantification) and $|\mathord{\cdot}|$ is a function (the interpretation function) whose domain is the signature of $\mathcal{L}$ such that $|c| \in M$, $|f^n|: M^n \rightarrow M$ and $|R^n|: W \rightarrow \mathcal{P}(M^n)$, subject to the persistence constraint that $w \prec u$ only if $|R^n|(w) \subseteq |R^n|(u)$. For a term $t(\overline{v}) \in \mathcal{L}$ and $\overline{a} \in M^n$, the denotation function $|t|(\overline{a})$ is defined recursively as follows:
\begin{align*}
|c|(\overline{a}) &= |c| \\
|v_i|(\overline{a}) &= a_i \\
|f^n(t_1,...,t_n)|(\overline{a}) &= |f^n|(|t_1|(\overline{a}),...,|t_n|(\overline{a})).
\end{align*}
For a formula $\phi(\overline{v}) \in \mathcal{L}, \overline{a} \in M^n$ and $w \in W$, the satisfaction relation $\mathfrak{M}, w \Vdash \phi(\overline{a})$ is defined recursively as follows (suppressing $\mathfrak{M}$ for brevity):
\begin{align*}
w &\Vdash \top(\overline{a}) \\
w &\not\Vdash \bot(\overline{a}) \\
w \Vdash R^n(t_1,...,t_n)(\overline{a}) &\iff \langle |t_1|(\overline{a}),...,|t_n|(\overline{a}) \rangle \in |R^n|(w) \\
w \Vdash (\phi \wedge \psi)(\overline{a}) &\iff w \Vdash \phi(\overline{a}) \text{ and } w \Vdash \psi(\overline{a}) \\
w \Vdash (\phi \vee \psi)(\overline{a}) &\iff w \Vdash \phi(\overline{a}) \text{ or } w \Vdash \psi(\overline{a}) \\
w \Vdash (\phi \rightarrow \psi)(\overline{a}) &\iff \text{for all } u \succ w: \text{if } u \Vdash \phi(\overline{a}) \text{ then } u \Vdash \psi(\overline{a}) \\
w \Vdash \exists v \phi(\overline{a}) &\iff \text{for some } b \in M: w \Vdash \phi(\overline{a}, b)  \\
w \Vdash \forall v \phi(\overline{a}) &\iff \text{for all } b \in M: w \Vdash \phi(\overline{a}, b). 
\end{align*}

\begin{theorem}[Persistence] If $w \Vdash \phi(\overline{a})$ and $w \prec u$ then $u \Vdash \phi(\overline{a})$.
\end{theorem}
\begin{proof}
An easy induction on the construction of $\mathcal{L}$-formulas.
\end{proof}

\noindent For sentences $\Gamma \cup \{\phi\} \subseteq \mathcal{L}$, we write $\Gamma \models \phi$ iff for every $\mathcal{L}$-model $\mathfrak{M}$ and every reflexive world $w \in \mathfrak{M}$: $w \Vdash \Gamma$ only if $w \Vdash \phi$. We will later show that $\textsf{TJK}^{d+}$ is the logic defined by $\models$. By restricting $\models$ to reflexive worlds, we preserve modus ponens. On the other hand, since a reflexive world may see an irreflexive world, we lose conditional proof. For example, $\phi \wedge (\phi \rightarrow \psi) \models \psi$ but $\not \models \phi \wedge (\phi \rightarrow \psi) \rightarrow \psi$. As we will see, however, $\models$ does validate a weaker form of conditional proof. The logic which results from $\models$ by dropping the requirement that $\models$ only be evaluated at reflexive worlds is known as constant domain basic logic ($\textsf{BQL}_\textsf{CD}$).\begin{footnote}{See Middleton (2020) for a discussion of $\textsf{BQL}_\textsf{CD}$.}\end{footnote} The use of `basic' here derives from Visser (1981), where the propositional fragment of $\textsf{BQL}_\textsf{CD}$ is defined, rather than $\textsf{B}^{d+}$.\begin{footnote}{The variable domain extension of Visser's basic propositional logic was studied by Ruitenburg (1998) under the name `basic predicate calculus'.}\end{footnote} Thus, by my preferred name for $\textsf{TJK}^{d+}$ is constant domain basic logic with a reflexive root ($\textsf{BQL}_\textsf{CD}^r$).\begin{footnote}{$\textsf{BQL}_\textsf{CD}^r$ was first discussed by Restall (1994) under the name `$bka$', though Restall was only able to prove completeness for the propositional fragment.}\end{footnote} 

\begin{theorem}[Compactness] If $\Gamma \models \phi$ then $\Gamma_0 \models \phi$ for some finite $\Gamma_0 \subseteq \Gamma$.
\end{theorem}
\begin{proof} Similar to the ultraproduct proof of compactness for classical first-order logic (see e.g.\ Poizat (2000)).
\end{proof}

\section{A simplified natural deduction system for $\textsf{TJK}^{d+}$}

In this section, I introduce a simplified natural deduction system for $\textsf{TJK}^{d+}$ ($\textsf{BQL}_\textsf{CD}^r$) in which the axioms for $\rightarrow$ are (mostly) replaced by a restricted form of conditional proof. Let $\Pi$ be a proof-tree with a leaf labelled by a possibly discharged occurrence $\phi^i$ of $\phi$.  We say $\phi^i$ is \textit{unsafe} in $\Pi$ iff $\phi^i$ occupies the following position:
$$
\infer*
	{}
	{
	....
	&
	\infer
		{\beta}
		{
		\infer*
			{\alpha}
			{}
		&
		\infer*
			{\alpha \rightarrow \beta}
			{\phi^i}
		}
	&
	....
	}
$$
(i.e.\ $\phi^i$ is a possibly discharged assumption in the right main subtree of an application of modus ponens). So, for example, the occurrence of $\alpha \rightarrow (\beta \rightarrow \gamma)$ in each of the following proofs is unsafe:
$$
\infer
	{\beta \rightarrow \gamma}
	{
	\alpha
	&
	\alpha \rightarrow (\beta \rightarrow \gamma)
	}
\qquad
\infer
	{\delta}
	{
	\alpha
	&
	\infer
		{\alpha \rightarrow \delta}
		{
		\alpha \rightarrow (\beta \rightarrow \gamma)
		&
		(\beta \rightarrow \gamma) \rightarrow \delta
		}
	}
$$
The simplified natural deduction system $\mathcal{N}\textsf{BQL}_\textsf{CD}^r$ for $\textsf{TJK}^{d+}$ ($\textsf{BQL}_\textsf{CD}^r$) consists of all trees of (possibly discharged) $\mathcal{L}^+$-sentences constructed in accordance with the following inference rules, where restrictions C1 -- C4 from \S2 remain in place and, in addition,
\begin{enumerate}[(C5)]
\item no unsafe occurrence of an open assumption may be discharged.
\end{enumerate}
Consequently, by C4 and C5, $\phi(a_i)$ may not occur unsafely in the right main subproof of $\exists$-Elim.

$$
[\top]\hspace{2mm}(\top\text{-Int})
\qquad
\infer[(\bot\text{-Elim})]
	{\phi}
	{\bot}
$$
$$
\infer[(\wedge\text{-Int})]
	{\phi \wedge \psi}
	{
	\phi
	&
	\psi
	}
\qquad
\infer[(\wedge\text{-Elim})]
	{\phi / \psi}
	{\phi \wedge \psi}
$$
$$
\infer[(\vee\text{-Int})]
	{\phi \vee \psi}
	{\phi / \psi}
\qquad
\infer[(\vee\text{-Elim})]
	{\chi}
	{
	\phi \vee \psi
	&
	\infer*
		{\chi}
		{[\phi]}
	&
	\infer*
		{\chi}
		{[\psi]}
	}
$$
$$
\infer[(\rightarrow\hspace{-1mm}\text{-Int})]
	{\phi \rightarrow \psi}
	{
	\infer*
		{\psi}
		{[\phi]}
	}
\qquad
\infer[(\rightarrow\hspace{-1mm}\text{-Elim})]
	{\psi}
	{
	\phi
	&
	\phi \rightarrow \psi
	}
$$
$$
\infer[(\text{Internal Transitivity})]
	{\phi \rightarrow \chi}
	{
	\phi \rightarrow \psi
	&
	\psi \rightarrow \chi
	}
$$
$$
\infer[(\text{Internal } \wedge\hspace{-1mm}\text{-Int})]
	{\phi \rightarrow \psi \wedge \chi}
	{
	\phi \rightarrow \psi
	&
	\phi \rightarrow \chi
	}
\qquad
\infer[(\text{Internal } \vee\hspace{-1mm}\text{-Elim})]
	{\phi \vee \psi \rightarrow \chi}
	{
	\phi \rightarrow \chi
	&
	\psi \rightarrow \chi
	}
$$
$$
\infer[(\text{Internal } \forall\text{-Int})]
	{\phi \rightarrow \forall v \psi}
	{\forall v (\phi \rightarrow \psi)}
\qquad
\infer[(\text{Internal } \exists\text{-Elim})]
	{\exists v \phi \rightarrow \psi}
	{\forall v (\phi \rightarrow \psi)}
$$
$$
\infer[(\forall\text{-Int})]
	{\forall v \phi}
	{\phi(a_i)}
\qquad
\infer[(\forall\text{-Elim})]
	{\phi(t)}
	{\forall v \phi}
\qquad
\infer[(\text{CD})]
	{\phi \vee \forall v \psi}
	{\forall v (\phi \vee \psi)}
$$
$$
\infer[(\exists\text{-Int})]
	{\exists v \phi}
	{\phi(t)}
\qquad
\infer[(\exists\text{-Elim})]
	{\psi}
	{
	\exists v \phi
	&
	\infer*
		{\psi}
		{[\phi(a_i)]}
	}
$$
We write $\Gamma \vdash \phi$ iff there exists a proof of $\phi$ from $\Gamma$ in $\mathcal{N}\textsf{BQL}_\textsf{CD}^r$. Note that the Curry derivation is not valid in $\mathcal{N}\textsf{BQL}_\textsf{CD}^r$, since the starred assumption is unsafe at the point where it is discharged:
$$
\infer
	{\bot}
	{
	\infer
		{T\ulcorner C \urcorner}
		{
		\infer
			{T\ulcorner C \urcorner \rightarrow \bot}
			{
			\infer
				{\bot}
				{
				[T\ulcorner C \urcorner]
				&
				\infer
					{T\ulcorner C \urcorner \rightarrow \bot}
					{
					^*[T\ulcorner C \urcorner]
					&
					T\ulcorner C \urcorner \rightarrow (T\ulcorner C \urcorner \rightarrow \bot)
					}
				}
			}
		&
		(T\ulcorner C \urcorner \rightarrow \bot) \rightarrow T\ulcorner C \urcorner
		}
	&
	\infer*
		{T\ulcorner C \urcorner \rightarrow \bot}
		{}
	}
$$

\subsection{Reduction theorem}

A key property of $\mathcal{N}\textsf{BQL}_\textsf{CD}^r$ is that $\rightarrow$-Elim is, in a sense to be made precise, eliminable from proofs. This property allows us to show both that $\mathcal{N}\textsf{BQL}_\textsf{CD}^r$ coincides with the original axiomatization of $\textsf{TJK}^{d+}$ and that $\mathcal{N}\textsf{BQL}_\textsf{CD}^r$ is sound and complete with respect to the binary Kripke semantics defined in \S3. Let $\mathcal{N}\textsf{BQL}_\textsf{CD}$ denote the natural deduction system which results from removing $\rightarrow$-Elim from $\mathcal{N}\textsf{BQL}_\textsf{CD}^r$.\begin{footnote}{This notation derives from the fact that $\mathcal{N}\textsf{BQL}_\textsf{CD}$ is sound and complete with respect to $\textsf{BQL}_\textsf{CD}$ (see Middleton (2020) and \S5.1 of this paper, where part of the proof in Middleton (2020) is simplified).}\end{footnote} We write $\Gamma \vdash_{[-1]} \phi$ iff there exists a proof of $\phi$ from $\Gamma$ in $\mathcal{N}\textsf{BQL}_\textsf{CD}$. Note that, since $\mathcal{N}\textsf{BQL}_\textsf{CD}$ does not contain $\rightarrow$-Elim, proofs in $\mathcal{N}\textsf{BQL}_\textsf{CD}$ do not contain unsafe occurrences of open assumptions. Consequently, open assumptions can be discharged unrestrictedly in $\mathcal{N}\textsf{BQL}_\textsf{CD}$. 

\begin{lemma}[Distribution] $\phi \wedge (\psi \vee \chi) \vdash_{[-1]} (\phi \wedge \psi) \vee (\phi \wedge \chi)$.
\end{lemma}

\begin{lemma}[Infinite Distribution] $\phi \wedge \exists v \psi \vdash_{[-1]} \exists v(\phi \wedge \psi)$.
\end{lemma}

\begin{lemma}[$\wedge$-Release] $\phi \wedge \psi \rightarrow \chi \vdash_{[-1]} \phi \rightarrow (\psi \rightarrow \chi)$.
\end{lemma}

\noindent Let $\Box \phi$ abbreviate $\top \rightarrow \phi$. 

\begin{lemma}[$\forall$-Embedding] $\forall v \Box^n \phi \vdash_{[-1]} \Box^n \forall v \phi$.
\end{lemma}
\begin{proof} By induction on $n$.  The base case $n = 0$ is trivial.  For the induction step:
$$
\infer[(\text{Internal Transitivity})]
	{\Box^{n + 1} \forall v \phi}
	{
	\infer[(\text{Internal } \forall\text{-Int})]
		{\top \rightarrow \forall v \Box^n \phi}
		{\forall v \Box^{n + 1} \phi}
	&
	\infer
		{\forall v \Box^n \phi \rightarrow \Box^n \forall v \phi}
		{
		\infer*[\hspace{1cm} \text{(induction hypothesis)}]
			{\Box^n \forall v \phi}
			{[\forall v \Box^n \phi]}
		}	
	}
$$
\end{proof}

\noindent Often we will want to append a proof of $\psi$ from $\{\phi_1,...,\phi_n\}$ to individual proofs of $\phi_1$,...,$\phi_n$:
$$
\infer*
	{\psi}
	{
	\infer*
		{\phi_1}
		{}
	&
	...
	&
	\infer*
		{\phi_n}
		{}
	}
$$
When $\forall$-Int or $\exists$-Elim is part of our proof system, chaining together proofs in this manner does not in general result in a valid proof, since the eigenvariable constraints may get violated. However, we can always avoid this problem by appropriately renaming the eigenvariables. 

\begin{lemma}[Regularity] If $\phi_1,...,\phi_m \vdash_{[-1]} \psi$ then $\Box^n\phi_1,...,\Box^n\phi_m \vdash_{[-1]} \Box^n\psi$.
\end{lemma}
\begin{proof} Suppose $\phi_1,...,\phi_m \vdash_{[-1]} \psi$. We prove the lemma by induction on $n$. The base case $n = 0$ is trivial. For the induction step:\begin{footnote}{Thanks to the anonymous referee at the RSL for simplifying this proof.}\end{footnote}
$$
\infer[(\text{Internal Transitivity})]
	{\Box^{n + 1} \psi}
	{
	\infer
		{\top \rightarrow \bigwedge_i \Box^n \phi_i}
		{
		\infer
			{\text{Internal $\wedge$-Ints}}
			{
			\Box^{n+1}\phi_1
			&
			...
			&
			\Box^{n+1}\phi_m
			}
		}
	&
	\infer
		{\bigwedge_i \Box^n \phi_i \rightarrow \Box^n \psi}
		{
		\infer*[\hspace{2cm} \text{(induction hypothesis)}]
			{\Box^n \psi}
			{
			\infer
				{\Box^n \phi_1}
				{
				\infer
					{\wedge\text{-Elims}}
					{[\bigwedge_i \Box^n \phi_i]}
				}
			&
			...
			&
			\infer
				{\Box^n \phi_m}
				{
				\infer
					{\wedge\text{-Elims}}
					{[\bigwedge_i \Box^n \phi_i]}
				}
			}
		}
	}
$$
\end{proof}

\noindent Let $\mathcal{N}\textsf{BQL}_\textsf{CD}[-1] = \mathcal{N}\textsf{BQL}_\textsf{CD}$ and, for $n \geq 0$, let $\mathcal{N}\textsf{BQL}_\textsf{CD}[n]$ denote the natural deduction system obtained by adding the rule
$$
\infer
	{\psi}
	{
	\phi
	&
	\fbox{
	\infer*[\mathcal{N}\textsf{BQL}_\textsf{CD}\text{$[m]$ for } m < n]
		{\phi \rightarrow \psi}
		{}
	}
	}
$$
to $\mathcal{N}\textsf{BQL}_\textsf{CD}$, maintaining restrictions C1 -- C5. This rule states that modus ponens may be applied if the right main subproof is contained in the system $\mathcal{N}\textsf{BQL}_\textsf{CD}[m]$ for some $m$ such that $-1 \leq m < n$. For example, the following proof belongs to $\mathcal{N}\textsf{BQL}_\textsf{CD}[n]$ for every $n \geq 1$:
$$
\infer
	{\alpha}
	{
	\infer
		{\chi}
		{
		\psi
		&
		\infer
			{\psi \rightarrow \chi}
			{
			[\top]
			& 
			\top \rightarrow (\psi \rightarrow \chi)
			}
		}
	&
	\chi \rightarrow \alpha
	}
$$
The reason for introducing the systems $\{\mathcal{N}\textsf{BQL}_\textsf{CD}[n] : n \in \omega\}$ is that, since proofs in $\mathcal{N}\textsf{BQL}_\textsf{CD}^r$ contain only finitely many applications of $\rightarrow$-Elim, every proof in $\mathcal{N}\textsf{BQL}_\textsf{CD}^r$ belongs to $\mathcal{N}\textsf{BQL}_\textsf{CD}[n]$ for some $n$. So we can show that proofs in $\mathcal{N}\textsf{BQL}_\textsf{CD}^r$ reduce to proofs in $\mathcal{N}\textsf{BQL}_\textsf{CD}$ by showing that proofs in each $\mathcal{N}\textsf{BQL}_\textsf{CD}[n]$ reduce to proofs in $\mathcal{N}\textsf{BQL}_\textsf{CD}$, which can be done by induction on $n$. We write $\Gamma \vdash_{[n]} \phi$ iff there exists a proof of $\phi$ from $\Gamma$ in $\mathcal{N}\textsf{BQL}_\textsf{CD}[n]$. We write $\Gamma : \Sigma \vdash_{[n]} \phi$ iff there exists a proof of $\phi$ from $\Gamma \cup \Sigma$ in $\mathcal{N}\textsf{BQL}_\textsf{CD}[n]$ such that every open assumption which occurs unsafely in the proof belongs to $\Gamma$ (open assumptions which occur unsafely in the proof may or may not belong to $\Sigma$). We define $\bigwedge \emptyset = \top$.

\begin{lemma}[Relative Deduction] For $|\Sigma| < \omega$, $n \geq 0$: if $\Gamma : \Sigma \vdash_{[n]} \phi$ then $\Gamma \vdash_{[-1]} \Box^n(\bigwedge \Sigma \rightarrow \phi)$.
\end{lemma}
\begin{proof} Suppose the lemma holds for all $m$ such that $0 \leq m < n$ (the outer induction hypothesis).  We prove by induction on the construction of proofs in $\mathcal{N}\textsf{BQL}_\textsf{CD}[n]$ that the lemma holds for $n$ (the inner induction hypothesis).

\underline{Inner Base Cases} Suppose we have a one-line proof in $\mathcal{N}\textsf{BQL}_\textsf{CD}[n]$ of $\phi$ from $\Gamma : \Sigma$.  There are three cases to consider.

\underline{Case 1} $\phi = \top$. Then
$$
\infer
	{\Box^n (\bigwedge \Sigma \rightarrow \top)}
	{
	\infer
		{\rightarrow\hspace{-1mm}\text{-Ints}}
		{[\top]}
	}
$$
is a proof of $\Box^n(\bigwedge \Sigma \rightarrow \phi)$ from $\Gamma$ in $\mathcal{N}\textsf{BQL}_\textsf{CD}$.

\underline{Case 2} $\phi \in \Gamma$. Then
$$
\infer
	{\Box^n (\bigwedge \Sigma \rightarrow \phi)}
	{
	\infer
		{\rightarrow\hspace{-1mm}\text{-Ints}}
		{\phi}
	}
$$
is a proof of $\Box^n(\bigwedge \Sigma \rightarrow \phi)$ from $\Gamma$ in $\mathcal{N}\textsf{BQL}_\textsf{CD}$.

\underline{Case 3} $\phi \in \Sigma$. Then
$$
\infer
	{\Box^n(\bigwedge \Sigma \rightarrow \phi)}
	{
	\infer
		{\rightarrow\hspace{-1mm}\text{-Ints}}
		{
		\infer
			{\bigwedge \Sigma \rightarrow \phi}
			{
			\infer
				{\phi}
				{
				\infer
					{\wedge\text{-Elims}}
					{[\bigwedge \Sigma]}	
				}
			}
		}
	}
$$
is a proof of $\Box^n(\bigwedge \Sigma \rightarrow \phi)$ from $\Gamma$ in $\mathcal{N}\textsf{BQL}_\textsf{CD}$.

\underline{Inner Induction Steps} There are seven cases to consider.

\underline{Case 1} Suppose we have a proof of the form
$$
\infer
	{\phi}
	{
	\infer*
		{\alpha}
		{\vspace{1mm} \Gamma : \Sigma}
	}
$$
in $\mathcal{N}\textsf{BQL}_\textsf{CD}[n]$, where the final inference is $\bot$-Elim, $\wedge$-Elim, $\vee$-Int, Internal $\forall$-Int, Internal $\exists$-Elim, $\forall$-Elim, CD or $\exists$-Int.  Then, by the inner induction hypothesis and Regularity applied to Internal Transitivity, we can find a proof of the form
$$
\infer*
	{\Box^n(\bigwedge \Sigma \rightarrow \phi)}
	{
	\infer*
		{\Box^n(\bigwedge \Sigma \rightarrow \alpha)}
		{\Gamma}
	&
	\infer
		{\Box^n(\alpha \rightarrow \phi)}
		{
		\infer
			{\rightarrow\hspace{-1mm}\text{-Ints}}
			{
			\infer
				{\alpha \rightarrow \phi}
				{
				\infer
					{\phi}
					{[\alpha]}
				}
			}
		}
	}
$$
in $\mathcal{N}\textsf{BQL}_\textsf{CD}$.

\underline{Case 2} Suppose we have a proof of the form
$$
\infer
	{\phi}
	{
	\infer*
		{\alpha}
		{\vspace{1mm} \Gamma : \Sigma}
	&
	\infer*
		{\beta}
		{\vspace{1mm} \Gamma : \Sigma}
	}
$$
in $\mathcal{N}\textsf{BQL}_\textsf{CD}[n]$, where the final inference is $\wedge$-Int, Internal Transitivity, Internal $\wedge$-Int or Internal $\vee$-Elim. Then, by the inner induction hypothesis and Regularity applied to Internal $\wedge$-Int and Internal Transitivity, we can find a proof of the form
$$
\infer*
	{\Box^n(\bigwedge \Sigma \rightarrow \phi)}
	{
	\infer*
		{\Box^n(\bigwedge \Sigma \rightarrow \alpha \wedge \beta)}
		{
		\infer*
			{\Box^n(\bigwedge \Sigma \rightarrow \alpha)}
			{\Gamma}
		&
		\infer*
			{\Box^n(\bigwedge \Sigma \rightarrow \beta)}
			{\Gamma}
		}
	&
	\infer
		{\Box^n(\alpha \wedge \beta \rightarrow \phi)}
		{
		\infer
			{\rightarrow\hspace{-1mm}\text{-Ints}}
			{
			\infer
				{\alpha \wedge \beta \rightarrow \phi}
				{
				\infer
					{\phi}
					{
					\infer
						{\alpha}
						{[\alpha \wedge \beta]}
					&
					\infer
						{\beta}
						{[\alpha \wedge \beta]}
					}
				}
			}
		}
	}
$$
in $\mathcal{N}\textsf{BQL}_\textsf{CD}$.

\underline{Case 3} Suppose we have a proof of the form
$$
\infer
	{\phi}
	{
	\infer*
		{\alpha \vee \beta}
		{\vspace{1mm} \Gamma : \Sigma}		
	&
	\infer*
		{\phi}
		{\Gamma : \Sigma, [\alpha]}		
	&
	\infer*
		{\phi}
		{\Gamma : \Sigma, [\beta]}	
	}
$$
in $\mathcal{N}\textsf{BQL}_\textsf{CD}[n]$. There are two subcases to consider.

\underline{Subcase 1} $\Sigma = \emptyset$.  Then, by the inner induction hypothesis and Regularity applied to Internal $\vee$-Elim and Internal Transitivity, we can find a proof of the form 
$$
\infer*
	{\Box^n(\top \rightarrow \phi)}
	{
	\infer*
		{\Box^n(\top \rightarrow \alpha \vee \beta)}
		{\Gamma}
	&
	\infer*
		{\Box^n(\alpha \vee \beta \rightarrow \phi)}
		{
		\infer*
			{\Box^n(\alpha \rightarrow \phi)}
			{\Gamma}
		&
		\infer*
			{\Box^n(\beta \rightarrow \phi)}
			{\Gamma}
		}
	}
$$
in $\mathcal{N}\textsf{BQL}_\textsf{CD}$.

\underline{Subcase 2} $\Sigma \neq \emptyset$.  Then, by Distribution, the inner induction hypothesis and Regularity applied to Internal $\vee$-Elim and Internal Transitivity, we can find a proof of the form
$$
\infer*
	{\Box^n(\bigwedge \Sigma \wedge (\alpha \vee \beta) \rightarrow \phi)}
	{
	\infer*
		{\Box^n(\bigwedge \Sigma \wedge (\alpha \vee \beta) \rightarrow (\bigwedge \Sigma \wedge \alpha) \vee (\bigwedge \Sigma \wedge \beta))}
		{\emptyset}
	&
	\infer*
		{\Box^n((\bigwedge \Sigma \wedge \alpha) \vee (\bigwedge \Sigma \wedge \beta) \rightarrow \phi)}
		{
		\infer*
			{\Box^n(\bigwedge \Sigma \wedge \alpha \rightarrow \phi)}
			{\Gamma}		
		&
		\infer*
			{\Box^n(\bigwedge \Sigma \wedge \beta \rightarrow \phi)}
			{\Gamma}
		}
	}
$$
in $\mathcal{N}\textsf{BQL}_\textsf{CD}$.  So, by the inner induction hypothesis and Regularity applied to Internal $\wedge$-Int and Internal Transitivity, we can find a proof of the form 
$$
\infer*
	{\Box^n(\bigwedge \Sigma \rightarrow \phi)}
	{
	\infer*
		{\Box^n(\bigwedge \Sigma \rightarrow \bigwedge \Sigma \wedge (\alpha \vee \beta))}
		{
		\infer*
			{\Box^n(\bigwedge \Sigma \rightarrow \bigwedge \Sigma)}
			{\emptyset}
		&
		\infer*
			{\Box^n(\bigwedge \Sigma \rightarrow \alpha \vee \beta)}
			{\Gamma}	
		}
	&
	\infer*
		{\Box^n(\bigwedge \Sigma \wedge (\alpha \vee \beta) \rightarrow \phi)}
		{\Gamma}
	}
$$
in $\mathcal{N}\textsf{BQL}_\textsf{CD}$. 

\underline{Case 4} Suppose we have a proof of the form
$$
\infer
	{\phi \rightarrow \psi}
	{
	\infer*
		{\psi}
		{\Gamma : \Sigma, [\phi]}
	}
$$
in $\mathcal{N}\textsf{BQL}_\textsf{CD}[n]$. There are two subcases to consider.

\underline{Subcase 1} $\Sigma = \emptyset$. Then, by the inner induction hypothesis, we can find a proof of the form
$$
\infer[(\rightarrow\hspace{-1mm}\text{-Int})]
	{\Box^n(\top \rightarrow (\phi \rightarrow \psi))}
	{
	\infer*
		{\Box^n(\phi \rightarrow \psi)}
		{\Gamma}
	}
$$
in $\mathcal{N}\textsf{BQL}_\textsf{CD}$. 

\underline{Subcase 2} $\Sigma \neq \emptyset$.  Then, by the inner induction hypothesis and Regularity applied to $\wedge$-Release, we can find a proof of the form
$$
\infer*
	{\Box^n(\bigwedge \Sigma \rightarrow (\phi \rightarrow \psi))}
	{
	\infer*
		{\Box^n(\bigwedge \Sigma \wedge \phi \rightarrow \psi)}
		{\Gamma}
	}
$$
in $\mathcal{N}\textsf{BQL}_\textsf{CD}$.

\underline{Case 5} Suppose we have a proof of the form
$$
\infer
	{\phi}
	{
	\infer*
		{\alpha}
		{\vspace{1mm} \Gamma : \Sigma}
	&
	\fbox{
	\infer*[\mathcal{N}\textsf{BQL}_\textsf{CD}\text{$[m]$ for } m < n]
		{\alpha \rightarrow \phi}
		{\vspace{1mm} \Gamma : \emptyset}
	}
	}
$$
in $\mathcal{N}\textsf{BQL}_\textsf{CD}[n]$. There are two subcases to consider.

\underline{Subcase 1} $m = -1$.  Then, by the inner induction hypothesis and Regularity applied to Internal Transitivity, we can find a proof of the form
$$
\infer*
	{\Box^n(\bigwedge \Sigma \rightarrow \phi)}
	{
	\infer*
		{\Box^n(\bigwedge \Sigma \rightarrow \alpha)}
		{\Gamma}
	&
	\infer
		{\Box^n(\alpha \rightarrow \phi)}
		{
		\infer
			{\rightarrow\hspace{-1mm}\text{-Ints}}
			{
			\infer*
				{\alpha \rightarrow \phi}
				{\Gamma}
			}
		}
	}
$$
in $\mathcal{N}\textsf{BQL}_\textsf{CD}$.

\underline{Subcase 2} $m \geq 0$. Then, by the \textit{outer} induction hypothesis, $\Gamma \vdash_{[-1]} \Box^{m}(\top \rightarrow (\alpha \rightarrow \phi))$. Hence, by the \textit{inner} induction hypothesis and Regularity applied to Internal Transitivity, we can find a proof of the form
$$
\infer*
	{\Box^n(\bigwedge \Sigma \rightarrow \phi)}
	{
	\infer*
		{\Box^n(\bigwedge \Sigma \rightarrow \alpha)}
		{\Gamma}
	&
	\infer
		{\Box^n(\alpha \rightarrow \phi)}
		{
		\infer
			{\rightarrow\hspace{-1mm}\text{-Ints}}
			{
			\infer*
				{\Box^{m+1}(\alpha \rightarrow \phi)}
				{\Gamma}
			}
		}
	}
$$
in $\mathcal{N}\textsf{BQL}_\textsf{CD}$.

\underline{Case 6} Suppose we have a proof of the form
$$
\infer
	{\forall v \phi}
	{
	\infer*
		{\phi(a_i)}
		{\vspace{1mm} \Gamma^* : \Sigma^*}
	}
$$
in $\mathcal{N}\textsf{BQL}_\textsf{CD}[n]$, where $\Gamma^* \subseteq \Gamma$, $\Sigma^* \subseteq \Sigma$ and $a_i$ does not occur in $\Gamma^* \cup \Sigma^* \cup \{\phi\}$. Then, by the inner induction hypothesis, $\forall$-Embedding and Regularity applied to Internal $\forall$-Int and Internal Transitivity, we can find a proof of the form
$$
\infer*
	{\Box^n(\bigwedge \Sigma \rightarrow \forall v \phi)}
	{
	\infer*
		{\Box^n(\bigwedge \Sigma \rightarrow \bigwedge \Sigma^*)}
		{\emptyset}
	&
	\infer*
		{\Box^n(\bigwedge \Sigma^* \rightarrow \forall v \phi)}
		{
		\infer*
			{\Box^n\forall v(\bigwedge \Sigma^* \rightarrow \phi)}
			{
			\infer
				{\forall v\Box^n(\bigwedge \Sigma^* \rightarrow \phi)}
				{
				\infer*
					{\Box^n(\bigwedge \Sigma^* \rightarrow \phi(a_i))}
					{\Gamma^*}
				}
			}
		}
	}
$$
in $\mathcal{N}\textsf{BQL}_\textsf{CD}$.

\underline{Case 7} Suppose we have a proof of the form
$$
\infer
	{\phi}
	{
	\infer*
		{\exists v \psi}
		{\vspace{1mm} \Gamma : \Sigma}
	&
	\infer*
		{\phi}
		{\Gamma^* : \Sigma^*, [\psi(a_i)]}
	}
$$
in $\mathcal{N}\textsf{BQL}_\textsf{CD}[n]$, where $\Gamma^* \subseteq \Gamma$, $\Sigma^* \subseteq \Sigma$ and $a_i$ does not occur in $\Gamma^* \cup \Sigma^* \cup \{\phi, \psi\}$. There are two subcases to consider.

\underline{Subcase 1} $\Sigma^* = \emptyset$.  Then, by the inner induction hypothesis, $\forall$-Embedding and Regularity applied to Internal $\exists$-Elim and Internal Transitivity, we can find a proof of the form
$$
\infer*
	{\Box^n(\bigwedge \Sigma \rightarrow \phi)}
	{
	\infer*
		{\Box^n(\bigwedge \Sigma \rightarrow \exists v \psi)}
		{\Gamma}
	&
	\infer*
		{\Box^n(\exists v \psi \rightarrow \phi)}
		{
		\infer*
			{\Box^n \forall v(\psi \rightarrow \phi)}
			{
			\infer
				{\forall v \Box^n(\psi \rightarrow \phi)}
				{
				\infer*
					{\Box^n(\psi(a_i) \rightarrow \phi)}
					{\Gamma^*}
				}
			}
		}
	}
$$
in $\mathcal{N}\textsf{BQL}_\textsf{CD}$.

\underline{Subcase 2} $\Sigma^* \neq \emptyset$.  Then, by the inner induction hypothesis, Infinite Distribution, $\forall$-Embedding and Regularity applied to Internal $\exists$-Elim and Internal Transitivity, we can find a proof of the form
$$
\infer*
	{\Box^n(\bigwedge \Sigma^* \wedge \exists v \psi \rightarrow \phi)}
	{
	\infer*
		{\Box^n(\bigwedge \Sigma^* \wedge \exists v \psi \rightarrow \exists v(\bigwedge \Sigma^* \wedge \psi))}
		{\emptyset}
	&
	\infer*
		{\Box^n(\exists v(\bigwedge \Sigma^* \wedge \psi) \rightarrow \phi)}
		{
		\infer*
			{\Box^n \forall v(\bigwedge \Sigma^* \wedge \psi \rightarrow \phi)}
			{
			\infer
				{\forall v \Box^n(\bigwedge \Sigma^* \wedge \psi \rightarrow \phi)}
				{
				\infer*
					{\Box^n(\bigwedge \Sigma^* \wedge \psi(a_i) \rightarrow \phi)}
					{\Gamma^*}
				}
			}
		}
	}
$$
in $\mathcal{N}\textsf{BQL}_\textsf{CD}$.  So, by Regularity applied to Internal Transitivity, we can find a proof of the form
$$
\infer*
	{\Box^n(\bigwedge \Sigma \wedge \exists v \psi \rightarrow \phi)}
	{
	\infer*
		{\Box^n(\bigwedge \Sigma \wedge \exists v \psi \rightarrow \bigwedge \Sigma^* \wedge \exists v \psi)}
		{\emptyset}
	&
	\infer*
		{\Box^n(\bigwedge \Sigma^* \wedge \exists v \psi \rightarrow \phi)}
		{\Gamma^*}
	}
$$
in $\mathcal{N}\textsf{BQL}_\textsf{CD}$. Hence, by the inner induction hypothesis and Regularity applied to Internal $\wedge$-Int and Internal Transitivity, we can find a proof of the form
$$
\infer*
	{\Box^n(\bigwedge \Sigma \rightarrow \phi)}
	{
	\infer*
		{\Box^n(\bigwedge \Sigma \rightarrow \bigwedge \Sigma \wedge \exists v \psi)}
		{
		\infer*
			{\Box^n(\bigwedge \Sigma \rightarrow \bigwedge \Sigma)}
			{\emptyset}
		&
		\infer*
			{\Box^n(\bigwedge \Sigma \rightarrow \exists v \psi)}
			{\Gamma}
		}
	&
	\infer*
		{\Box^n(\bigwedge \Sigma \wedge \exists v \psi \rightarrow \phi)}
		{\Gamma^*}
	}
$$
in $\mathcal{N}\textsf{BQL}_\textsf{CD}$. 
\end{proof}

\begin{theorem}[Reduction] If $\Gamma \vdash \phi$ then $\Gamma \vdash_{[-1]} \Box^n\phi$ for some $n$.
\end{theorem}
\begin{proof} Suppose $\Gamma \vdash \phi$.  Then, since proofs in $\mathcal{N}\textsf{BQL}_\textsf{CD}^r$ contain at most finitely many applications of $\rightarrow$-Elim, $\Gamma : \emptyset \vdash_{[n]} \phi$ for some $n \geq -1$. If $n = -1$ then we're done. Suppose $n \geq 0$. Then, by Relative Deduction, $\Gamma \vdash_{[-1]} \Box^{n + 1} \phi$. 
\end{proof}

\begin{theorem}[Unrestricted $\vee$-Elim] Suppose (i) $\Gamma \vdash \phi \vee \psi$, (ii) $\Gamma, \phi \vdash \chi$ and (iii) $\Gamma, \psi \vdash \chi$. Then $\Gamma \vdash \chi$.
\end{theorem}
\begin{proof} By Reduction: $\Gamma, \phi \vdash_{[-1]} \Box^n \chi$ and $\Gamma, \psi \vdash_{[-1]} \Box^m \chi$ for some $n$, $m$. Suppose without loss that $n \leq m$. Then, by repeated applications of $\rightarrow$-Int, we get $\Gamma, \phi \vdash_{[-1]} \Box^m \chi$. Since all open occurrences of $\phi$ and $\psi$ in the witnessing proofs are safe, an application of $\vee$-Elim gives $\Gamma \vdash \Box^m \chi$. Repeated applications of $\top$-Int and $\rightarrow$-Elim then give $\Gamma \vdash \chi$.
\end{proof}

\begin{theorem}[Unrestricted $\exists$-Elim] Suppose (i) $\Gamma \vdash \exists v \phi$ and (ii) $\Sigma, \phi(a_i) \vdash \psi$ for $a_i \not \in \Sigma \cup \{\phi, \psi\}$. Then $\Gamma, \Sigma \vdash \psi$.
\end{theorem}
\begin{proof} By Reduction: $\Sigma, \phi(a_i) \vdash_{[-1]} \Box^n \psi$ for some $n$. Since all open occurrences of $\phi(a_i)$ in the witnessing proof are safe, an application of $\exists$-Elim gives $\Gamma, \Sigma \vdash \Box^n \psi$. Repeated applications of $\top$-Int and $\rightarrow$-Elim then give $\Gamma, \Sigma \vdash \psi$. 
\end{proof}

\subsection{$\textsf{BQL}_\textsf{CD}^r = \textsf{TJK}^{d+}$}

We now show that $\mathcal{N}\textsf{BQL}_\textsf{CD}^r$ is in fact equivalent to the original axiomatization of $\textsf{TJK}^{d+}$ (see \S2). We write $\Gamma \vdash_{\textsf{TJK}^{d+}} \phi$ iff $\phi$ is derivable from $\Gamma$ in the original axiomatization.

\begin{lemma} $\phi \wedge \psi \rightarrow \chi \vdash_{\textsf{TJK}^{d+}}\phi \rightarrow (\psi \rightarrow \chi)$.
\end{lemma}
\begin{proof}
$$
\infer*
	{\phi \rightarrow (\psi \rightarrow \phi \wedge \psi)}
	{
	\infer*
		{\phi \rightarrow (\psi \rightarrow \phi) \wedge (\psi \rightarrow \psi)}
		{
		[\phi \rightarrow (\psi \rightarrow \phi)]
		&
		\infer
			{\phi \rightarrow (\psi \rightarrow \psi)}
			{
			[\psi \rightarrow \psi]
			&
			[(\psi \rightarrow \psi) \rightarrow (\phi \rightarrow (\psi \rightarrow \psi))]
			}
		}
	}
$$

$$
\infer*
	{\phi \rightarrow (\psi \rightarrow \chi)}
	{
	\infer*
		{\phi \rightarrow (\psi \rightarrow \phi \wedge \psi)}
		{}
	&
	\infer
		{(\psi \rightarrow \phi \wedge \psi) \rightarrow (\psi \rightarrow \chi)}
		{
		\phi \wedge \psi \rightarrow \chi
		&
		[(\phi \wedge \psi \rightarrow \chi) \rightarrow ((\psi \rightarrow \phi \wedge \psi) \rightarrow (\psi \rightarrow \chi))]
		}
	}
$$
\end{proof}

\begin{lemma} For $|\Gamma| < \omega$: if $\Gamma \vdash_{[-1]} \phi$ then $\vdash_{\textsf{TJK}^{d+}} \bigwedge \Gamma \rightarrow \phi$.
\end{lemma}
\begin{proof} By induction on the construction of proofs in $\mathcal{N}\textsf{BQL}_\textsf{CD}$. The only non-obvious step is $\rightarrow$-Int, where we appeal to the previous lemma.
\end{proof}

\begin{theorem} $\Gamma \vdash \phi$ iff $\Gamma \vdash_{\textsf{TJK}^{d+}} \phi$.
\end{theorem}
\begin{proof} $\impliedby$ By induction on the construction of proofs in the original axiomatization of $\textsf{TJK}^{d+}$, appealing to Unrestricted $\vee$-Elim and Unrestricted $\exists$-Elim in the induction steps for $\vee$-Elim and $\exists$-Elim respectively.

$\implies$ Suppose $\Gamma \vdash \phi$. Then $\Gamma_0 \vdash \phi$ for finite $\Gamma_0 \subseteq \Gamma$. By Reduction, $\Gamma_0 \vdash_{[-1]} \Box^n \phi$ for some $n$. So, by the previous lemma, $\vdash_{\textsf{TJK}^{d+}} \bigwedge \Gamma_0 \rightarrow \Box^n \phi$. Hence $\Gamma_0 \vdash_{\textsf{TJK}^{d+}} \phi$.\begin{footnote}{Thanks to Andrew Bacon for suggesting this argument.}\end{footnote}
\end{proof}

$\textsf{TJK}^+$, first studied by Bacon (2013a), is the logic obtained by removing $\exists$-Elim from the original axiomatization of $\textsf{TJK}^{d+}$. Write $\Gamma \vdash_{\textsf{TJK}^+} \phi$ iff there exists a proof of $\phi$ from $\Gamma$ in $\textsf{TJK}^+$. Since we did not rely on $\exists$-Elim in the proof of Lemma 8, we have effectively shown that for $|\Gamma| < \omega$, $\Gamma \vdash_{[-1]} \phi$ only if $\vdash_{\textsf{TJK}^+} \bigwedge \Gamma \rightarrow \phi$. This allows us to prove that $\textsf{TJK}^+ = \textsf{TJK}^{d+}$.

\begin{proposition} $\textsf{TJK}^+ = \textsf{TJK}^{d+}$.
\end{proposition}
\begin{proof} Trivially, $\textsf{TJK}^+ \subseteq \textsf{TJK}^{d+}$. For the converse inclusion, we have
\begin{align*}
\Gamma \vdash_{\textsf{TJK}^{d+}} \phi &\implies \Gamma_0 \vdash_{\textsf{TJK}^{d+}} \phi \text{ for finite } \Gamma_0 \subseteq \Gamma \\ 
&\implies \Gamma_0 \vdash \phi \\
&\implies \Gamma_0 \vdash_{[-1]} \Box^n\phi \text{ for some } n \\
&\implies \vdash_{\textsf{TJK}^+} \textstyle\bigwedge \Gamma_0 \rightarrow \Box^n \phi \\
&\implies \Gamma_0 \vdash_{\textsf{TJK}^+} \phi.
\end{align*}
\end{proof}

\section{Soundness and completeness}

In this section, we show that $\mathcal{N}\textsf{BQL}_\textsf{CD}^r$ is sound and complete with respect to the binary Kripke semantics outlined in \S3. We write $\Gamma : \Sigma \models \phi$ iff for every $\mathcal{L}$-model $\mathfrak{M}$, every reflexive world $w \in \mathfrak{M}$ and every $u \succ w$: if $w \Vdash \Gamma$ and $u \Vdash \Sigma$ then $u \Vdash \phi$.

\begin{lemma}[Generalized Soundness] If $\Gamma : \Sigma \vdash \phi$ then $\Gamma : \Sigma \models \phi$.
\end{lemma}
\begin{proof} By induction on the construction of proofs in $\mathcal{N}\textsf{BQL}_\textsf{CD}^r$. The base case is easy. The induction steps are also easy except for $\rightarrow$-Elim and the rules discharging assumptions.

\underline{$\rightarrow$-Elim} Suppose we have a proof of the form 
$$
\infer
	{\psi}
	{
	\infer*
		{\phi}
		{\vspace{1mm} \Gamma : \Sigma}
	&
	\infer*
		{\phi \rightarrow \psi}
		{\vspace{1mm} \Gamma : \emptyset}
	}
$$ 
in $\mathcal{N}\textsf{BQL}_\textsf{CD}^r$. By the induction hypothesis, $\Gamma : \Sigma \models \phi$ and $\Gamma : \emptyset \models \phi \rightarrow \psi$. Let $w \prec u$ for reflexive $w$. Suppose $w \Vdash \Gamma$ and $u \Vdash \Sigma$. Then $w \Vdash \phi \rightarrow \psi$ and $u \Vdash \phi$. So $u \Vdash \psi$.

\underline{$\rightarrow$-Int} Suppose we have a proof of the form
$$
\infer
	{\phi \rightarrow \psi}
	{
	\infer*
		{\psi}
		{\Gamma : \Sigma, [\phi]}
	}
$$
in $\mathcal{N}\textsf{BQL}_\textsf{CD}^r$. By the induction hypothesis, $\Gamma : \Sigma, \phi \models \psi$. Let $w \prec u$ for reflexive $w$. Suppose $w \Vdash \Gamma$ and $u \Vdash \Sigma$. Let $z \Vdash \phi$ for $z \succ u$. Then, by Persistence, $z \Vdash \Sigma$. Also, by transitivity, $w \prec z$. So $z \Vdash \psi$. Hence $u \Vdash \phi \rightarrow \psi$. 

\underline{$\vee$-Elim} Suppose we have a proof of the form
$$
\infer
	{\chi}
	{
	\infer*
		{\phi \vee \psi}
		{\vspace{1mm} \Gamma : \Sigma}
	&
	\infer*
		{\chi}
		{\Gamma : \Sigma, [\phi]}
	&
	\infer*
		{\chi}
		{\Gamma : \Sigma, [\psi]}
	}
$$
in $\mathcal{N}\textsf{BQL}_\textsf{CD}^r$. By the induction hypothesis, $\Gamma : \Sigma \models \phi \vee \psi$, $\Gamma : \Sigma, \phi \models \chi$ and $\Gamma : \Sigma, \psi \models \chi$. Let $w \prec u$ for reflexive $w$. Suppose $w \Vdash \Gamma$ and $u \Vdash \Sigma$. Then $u \Vdash \phi \vee \psi$. So $u \Vdash \phi$ or $u \Vdash \psi$. In either case, $u \Vdash \chi$.

\underline{$\exists$-Elim} Suppose we have a proof of the form
$$
\infer
	{\psi}
	{
	\infer*
		{\exists v \phi}
		{\vspace{1mm} \Gamma : \Sigma}
	&
	\infer*
		{\psi}
		{\Gamma^* : \Sigma^*, [\phi(a_i)]}
	}
$$
in $\mathcal{N}\textsf{BQL}_\textsf{CD}^r$, where $\Gamma^* \subseteq \Gamma$, $\Sigma^* \subseteq \Sigma$ and $a_i$ does not occur in $\Gamma^* \cup \Sigma^* \cup \{\phi, \psi\}$. By the induction hypothesis, $\Gamma : \Sigma \models \exists v \phi$ and $\Gamma^* : \Sigma^*, \phi(a_i) \models \psi$. Let $w \prec u$ in $\mathfrak{M}$ for reflexive $w$. Suppose $w \Vdash \Gamma$ and $u \Vdash \Sigma$. Then $u \Vdash \exists v \phi$. So $u \Vdash \phi(b)$ for some $b \in dom(\mathfrak{M})$. Let $\mathfrak{M}[a_i:b]$ denote the $\mathcal{L}^+$-model obtained from $\mathfrak{M}$ by setting $|a_i| = b$. Then $\mathfrak{M}[a_i:b], w \Vdash \Gamma^*$ and $\mathfrak{M}[a_i:b], u \Vdash \Sigma^* \cup \{\phi(a_i)\}$. So $\mathfrak{M}[a_i:b], u \Vdash \psi$. Hence $\mathfrak{M}, u \Vdash \psi$.
\end{proof}

\begin{corollary}[Soundness] If $\Gamma \vdash \phi$ then $\Gamma \models \phi$.
\end{corollary}
\begin{proof} Suppose $\Gamma \vdash \phi$. Then $\Gamma : \emptyset \vdash \phi$. So, by generalized soundness, $\Gamma : \emptyset \models \phi$. But then $\Gamma \models \phi$.
\end{proof}

\subsection{The canonical model}

The canonical model for $\textsf{BQL}_\textsf{CD}^r$ is identical to the canonical model for $\textsf{BQL}_\textsf{CD}$ (defined in Middleton (2020)). I repeat the definition here for ease of reference.  In order to prove that the canonical model behaves correctly, we need to temporarily assume that $\mathcal{L}$ is countable. However, this does not result in a loss of generality, since we can use compactness to leverage up our proof of completeness to $\mathcal{L}$ of arbitrary cardinality. A set of sentences $\Gamma \subseteq \mathcal{L}^+$ is called a prime saturated $\textsf{BQL}_\textsf{CD}$-theory iff $\Gamma$ satisfies the following properties:
\begin{align*}
&(\text{consistency}) &&\bot \not \in \Gamma \\
&(\textsf{BQL}_\textsf{CD}\text{-closure}) &&\text{if } \Gamma \vdash_{[-1]} \phi \text{ then } \phi \in \Gamma \\
&(\text{disjunction property}) &&\text{if } \phi \vee \psi \in \Gamma \text{ then } \phi \in \Gamma \text{ or } \psi \in \Gamma \\
&(\text{existence property}) &&\text{if } \exists v \phi \in \Gamma \text{ then } \phi(t) \in \Gamma \text{ for some } t \in \mathcal{L}^+ \\
&(\text{totality property}) &&\text{if } \phi(t) \in \Gamma \text{ for every } t \in \mathcal{L}^+ \text{ then } \forall v \phi \in \Gamma.
\end{align*}
Let $Sat(\textsf{BQL}_\textsf{CD})$ denote the set of prime saturated $\textsf{BQL}_\textsf{CD}$-theories. The canonical model $\mathfrak{C}$ is the $\mathcal{L}^+$-model $\langle Sat(\textsf{BQL}_\textsf{CD}), \prec, C, |\mathord{\cdot}| \rangle$ such that (i) $\Gamma \prec \Sigma$ iff for all $\phi, \psi$: if $\phi \rightarrow \psi \in \Gamma$ and $\phi \in \Sigma$ then $\psi \in \Sigma$, (ii) $C$ is the set of closed $\mathcal{L}^+$-terms and (iii) we have:
\begin{align*}
|c| &= c \\
|f^n|(t_1,...,t_n) &= f^n(t_1,...,t_n) \\
|R^n|(\Gamma) &= \{\langle t_1,...,t_n \rangle: R^n(t_1,...,t_n) \in \Gamma\}.
\end{align*}

\begin{proposition} $\mathfrak{C}$ is an $\mathcal{L}^+$-model.
\end{proposition}
\begin{proof} The set of sentences true at an arbitrary world in an $\mathcal{L}^+$-model is closed under $\textsf{BQL}_\textsf{CD}$. Thus, a world in an $\mathcal{L}^+$-model in which every element of the domain is named gives $Sat(\textsf{BQL}_\textsf{CD}) \neq \emptyset$. For transitivity, suppose $\Gamma \prec \Sigma \prec \Delta$, $\phi \rightarrow \psi \in \Gamma$ and $\phi \in \Delta$. Then $\top \rightarrow (\phi \rightarrow \psi) \in \Gamma$ and hence $\phi \rightarrow \psi \in \Sigma$. So $\psi \in \Delta$. For the persistence constraint, suppose $\Gamma \prec \Sigma$ and $\langle t_1,...,t_n \rangle \in |R^n|(\Gamma)$. Then $R^n(t_1,...,t_n) \in \Gamma$. So $\top \rightarrow R^n(t_1,...,t_n) \in \Gamma$ and hence $R^n(t_1,...,t_n) \in \Sigma$. So $\langle t_1,...,t_n \rangle \in |R^n|(\Sigma)$.
\end{proof}

\noindent The next lemma supercedes Lemma 7 in Middleton (2020), which unnecessarily appealed to Relative Deduction.\begin{footnote}{Furthermore, Relative Deduction is formulated incorrectly in Middleton (2020). To get the correct formulation, replace $\Sigma', \Gamma \vdash_{\Sigma} \phi$ for $\Sigma' \subseteq \Sigma$ with $\Sigma : \Gamma \vdash_{[0]} \phi$.}\end{footnote}

\begin{lemma}[Relative Extension] For $\Gamma \in Sat(\textsf{BQL}_\textsf{CD})$: if $\Gamma \not \vdash_{[-1]} \phi \rightarrow \psi$ then there exists $\Sigma \in Sat(\textsf{BQL}_\textsf{CD})$ such that $\Gamma \prec \Sigma$, $\phi \in \Sigma$ and $\psi \not \in \Sigma$.
\end{lemma}
\begin{proof} Suppose $\Gamma \not \vdash_{[-1]} \phi \rightarrow \psi$ and fix an enumeration $\{\alpha_i\}_{i \in \omega}$ of $\mathcal{L}^+$-sentences (recall $\mathcal{L}$ is now assumed to be countable). We first inductively define an increasing sequence $\{(\Sigma_n, \Delta_n) : n \in \omega\}$ of pairs of finite sets of $\mathcal{L}^+$-sentences $(\Sigma_n, \Delta_n)$ such that $\Gamma \not \vdash_{[-1]} \bigwedge \Sigma_n \rightarrow \bigvee \Delta_n$. For the base case, define $\Sigma_0 = \{\phi\}$ and $\Delta_0 = \{\psi\}$. For the induction step, suppose $\Sigma_n$, $\Delta_n$ have already been defined and $\Gamma \not \vdash_{[-1]} \bigwedge \Sigma_n \rightarrow \bigvee \Delta_n$. There are two cases to consider.

\underline{Case 1} $\Gamma \vdash_{[-1]} \bigwedge \Sigma_n \wedge \alpha_n \rightarrow \bigvee \Delta_n$. Then we define $\Sigma_{n + 1} = \Sigma_n$. First suppose $\alpha_n \neq \forall v \chi$. Then we define $\Delta_{n + 1} = \Delta_n \cup \{\alpha_n\}$. Suppose for a reductio that $\Gamma \vdash_{[-1]} \bigwedge \Sigma_n \rightarrow \bigvee \Delta_n \vee \alpha_n$. Then, by Internal $\wedge$-Int and Distribution, $\Gamma \vdash_{[-1]} \bigwedge \Sigma_n \rightarrow (\bigwedge \Sigma_n \wedge \bigvee \Delta_n) \vee (\bigwedge \Sigma_n \wedge \alpha_n)$. So, by Internal $\vee$-Elim, $\Gamma \vdash_{[-1]} \bigwedge \Sigma_n \rightarrow \bigvee \Delta_n$, which contradicts the induction hypothesis.

Suppose, on the other hand, that $\alpha_n = \forall v \chi$. By the same argument, $\Gamma \not \vdash_{[-1]} \bigwedge \Sigma_n \rightarrow \bigvee \Delta_n \vee \forall v \chi$. Suppose for a reductio that $\Gamma \vdash_{[-1]} \bigwedge \Sigma_n \rightarrow \bigvee \Delta_n \vee \chi(t)$ for every $t \in \mathcal{L}^+$. Then, since $\Gamma \in Sat(\textsf{BQL}_\textsf{CD})$, $\Gamma \vdash_{[-1]} \forall v(\bigwedge \Sigma_n \rightarrow \bigvee \Delta_n \vee \chi)$. But then, by Internal $\forall$-Int and CD, $\Gamma \vdash_{[-1]} \bigwedge \Sigma_n \rightarrow \bigvee \Delta_n \vee \forall v \chi$, which is a contradiction. So we can define $\Delta_{n+1} = \Delta_n \cup \{\forall v \chi, \chi(t)\}$ for some $t \in \mathcal{L}^+$ such that $\Gamma \not \vdash_{[-1]} \bigwedge \Sigma_n \rightarrow \bigvee \Delta_n \vee \forall v \chi \vee \chi(t)$.

\underline{Case 2} $\Gamma \not \vdash_{[-1]} \bigwedge \Sigma_n \wedge \alpha_n \rightarrow \bigvee \Delta_n$. Then we define $\Delta_{n + 1} = \Delta_n$. If $\alpha_n \neq \exists v \chi$ then we define $\Sigma_{n +1} = \Sigma_n \cup \{\alpha_n\}$. Suppose, on the other hand, that $\alpha_n = \exists v \chi$. Suppose for a reductio that $\Gamma \vdash_{[-1]} \bigwedge \Sigma_n \wedge \chi(t) \rightarrow \bigvee \Delta_n$ for every $t \in \mathcal{L}^+$. Then, since $\Gamma \in Sat(\textsf{BQL}_\textsf{CD})$, $\Gamma \vdash_{[-1]} \forall v(\bigwedge \Sigma_n \wedge \chi \rightarrow \bigvee \Delta_n)$. So, by Internal $\exists$-Elim and Infinite Distribution, $\Gamma \vdash_{[-1]} \bigwedge \Sigma_n \wedge \exists v \chi \rightarrow \bigvee \Delta_n$, which is a contradiction. So we can define $\Sigma_{n+1} = \Sigma_n \cup \{\exists v \chi, \chi(t)\}$ for some $t \in \mathcal{L}^+$ such that $\Gamma \not \vdash_{[-1]} \bigwedge \Sigma_n \wedge \exists v \chi \wedge \chi(t) \rightarrow \bigvee \Delta_n$.

Now define $\Sigma = \bigcup_{n \in \omega} \Sigma_n$. Clearly, $\phi \in \Sigma$. Furthermore, $\psi \not \in \Sigma$, for otherwise $\psi \in \Sigma_n$ for some $n$ and so $\Gamma \vdash_{[-1]} \bigwedge \Sigma_n \rightarrow \bigvee \Delta_n$. It is straightforward to verify $\Sigma \in Sat(\textsf{BQL}_\textsf{CD})$ using similar arguments. To verify $\Gamma \prec \Sigma$, suppose $\alpha_n \rightarrow \alpha_m \in \Gamma$ and $\alpha_n \in \Sigma$.  Then $\alpha_n \in \Sigma_k$ for some $k$ and so $\Gamma \vdash_{[-1]} \bigwedge \Sigma_k \rightarrow \alpha_m$. Suppose for a reductio that $\alpha_m \not \in \Sigma$. Then $\alpha_m \in \Delta_{m + 1}$. So $\Gamma \vdash_{[-1]} \bigwedge \Sigma_{\max\{k, m + 1\}} \rightarrow \bigvee \Delta_{\max\{k, m + 1\}}$, which is a contradiction.
\end{proof}

\begin{lemma}[Truth] $\mathfrak{C}, \Gamma \Vdash \phi$ iff $\phi \in \Gamma$.
\end{lemma}
\begin{proof} By induction on the complexity of $\mathcal{L}^+$-sentences, appealing to Relative Extension in the induction step for $\rightarrow$.
\end{proof}

\subsection{Completeness}

Continue to suppose $\mathcal{L}$ is countable. Define $Sat(\textsf{BQL}_\textsf{CD}^r)$ analogously to $Sat(\textsf{BQL}_\textsf{CD})$. Since the elements of $Sat(\textsf{BQL}_\textsf{CD}^r)$ are closed under modus ponens and $Sat(\textsf{BQL}_\textsf{CD}^r) \subseteq Sat(\textsf{BQL}_\textsf{CD})$, the elements of $Sat(\textsf{BQL}_\textsf{CD}^r)$ are reflexive worlds in the canonical model.

\begin{lemma}[Extension] For $\Gamma$ such that $|\{i: a_i \not \in \Gamma\}| = \omega$: if $\Gamma \not \vdash \phi$ then there exists $\Gamma^* \supseteq \Gamma$ such that $\Gamma^* \in Sat(\textsf{BQL}_\textsf{CD}^r)$ and $\phi \not \in \Gamma^*$.
\end{lemma}
\begin{proof} We can use Unrestricted $\vee$-Elim and Unrestricted $\exists$-Elim to run a similar argument to the proof of the Belnap Extension Lemma (see e.g.\ Priest (2002) \S6.2), with witnesses drawn from $\{a_i : i \in \omega\}$. Given that $\mathcal{L}$ is countable, the assumption that $|\{i: a_i \not \in \Gamma\}| = \omega$ ensures we never run out of witnesses.
\end{proof}

\noindent We now drop the assumption that $\mathcal{L}$ is countable. We prove that completeness holds over the extended language $\mathcal{L}^+$.

\begin{lemma}[Weak Completeness] For $|\Gamma| < \omega$: if $\Gamma \models \phi$ then $\Gamma \vdash \phi$. 
\end{lemma}
\begin{proof} Suppose $\Gamma \not \vdash \phi$.  Since $|\Gamma| < \omega$, we can find a countable first-order language $\mathcal{L}_0 \subseteq \mathcal{L}$ such that $\Gamma \cup \{\phi\} \subseteq \mathcal{L}_0^+$.  A forteriori, there does not exist a proof in $\mathcal{N}\textsf{BQL}_\textsf{CD}^r \upharpoonright \mathcal{L}_0^+$ of $\phi$ from $\Gamma$. Since $|\{i: a_i \not \in \Gamma\}| = \omega$, Extension gives $\Gamma^* \supseteq \Gamma$ such that $\Gamma^* \in Sat(\textsf{BQL}_\textsf{CD}^r)$ (where $Sat(\textsf{BQL}_\textsf{CD}^r)$ is defined over $\mathcal{L}_0^+$) and $\phi \not \in \Gamma^*$.  Let $\mathfrak{C}$ denote the canonical model over $\mathcal{L}_0^+$. Then, by Truth: $\mathfrak{C}, \Gamma^* \Vdash \Gamma$ and $\mathfrak{C}, \Gamma^* \not \Vdash \phi$.  Since $\Gamma^* \prec \Gamma^*$, an arbitrary expansion of $\mathfrak{C}$ to $\mathcal{L}^+$ gives $\Gamma \not \models \phi$. 
\end{proof}

\begin{theorem}[Completeness] If $\Gamma \models \phi$ then $\Gamma \vdash \phi$.
\end{theorem}
\begin{proof} Immediate from compactness and weak completeness.
\end{proof}

\subsection{Disjunction and existence Properties}

We can also use the canonical model to show that $\textsf{BQL}_\textsf{CD}^r$ satisfies the disjunction and existence properties over the base language $\mathcal{L}$.

\begin{lemma}[Intersection] For $I \neq \emptyset$, let $\{w\} \cup \{u_i\}_{i \in I} \subseteq \mathfrak{M}$ be such that (i) $|R^n|(w) = \bigcap_{i \in I} |R^n|(u_i)$, (ii) every $u_i$ is reflexive, (iii) $w \prec u_i$ for every $i$ and (iv) if $w \prec z$ and $z \neq w$ then $u_i \prec z$ for some $i$. Then, for $\phi(\overline{v}) \in \mathcal{L} \setminus \{\vee, \exists\}$: $w \Vdash \phi(\overline{a})$ iff for all $i$: $u_i \Vdash \phi(\overline{a})$.
\end{lemma}
\begin{proof} By induction on the construction of $\mathcal{L} \setminus \{\vee, \exists\}$-formulas. The base case is easy.  The induction steps are also easy except for $\rightarrow$.

\underline{$\rightarrow$} $\implies$ This direction follows from (iii) and Persistence.

$\impliedby$ Suppose $u_i \Vdash (\phi \rightarrow \psi)(\overline{a})$ for all $i$. Suppose for a reductio that $w \not \Vdash (\phi \rightarrow \psi)(\overline{a})$. Then $z \Vdash \phi(\overline{a})$ and $z \not \Vdash \psi(\overline{a})$ for some $z \succ w$. There are two cases to consider.

\underline{Case 1} $z \neq w$. Then, by (iv), $u_i \prec z$ for some $i$.  So $u_i \not \Vdash (\phi \rightarrow \psi)(\overline{a})$, which is a contradiction.

\underline{Case 2} $z = w$.  Then, by the induction hypothesis, $u_i \Vdash \phi(\overline{a})$ for all $i$ and $u_j \not \Vdash \psi(\overline{a})$ for some $j$.  So, by (ii), $u_j \not \Vdash (\phi \rightarrow \psi)(\overline{a})$, which is a contradiction.
\end{proof}

\begin{lemma}[Weak Disjunction Property] For $\Gamma \subseteq \mathcal{L} \setminus \{\vee, \exists\}$ such that $|\Gamma| \leq \omega$: if $\Gamma \models \phi \vee \psi$ then $\Gamma \models \phi$ or $\Gamma \models \psi$.
\end{lemma}
\begin{proof} Similar to the proof of the weak disjunction property in Middleton (2020).
\end{proof}

\begin{theorem}[Disjunction Property] For $\Gamma \subseteq \mathcal{L} \setminus \{\vee, \exists\}$: if $\Gamma \models \phi \vee \psi$ then $\Gamma \models \phi$ or $\Gamma \models \psi$.
\end{theorem}
\begin{proof} Immediate from compactness and the weak disjunction property. 
\end{proof}

\begin{lemma}[Weak Existence Property] Suppose $\mathcal{L}$ contains at least one constant symbol. Then, for $\Gamma \subseteq \mathcal{L} \setminus \{\vee, \exists\}$ such that $|\Gamma| \leq \omega$: $\Gamma \models \exists v \phi$ only if $\Gamma \models \phi(t)$ for some $t \in \mathcal{L}$.
\end{lemma}
\begin{proof} Similar to the proof of the weak existence property in Middleton (2020).
\end{proof}

\begin{theorem}[Existence Property] Suppose $\mathcal{L}$ contains at least one constant symbol. Then, for $\Gamma \subseteq \mathcal{L} \setminus \{\vee, \exists\}$: $\Gamma \models \exists v \phi$ only if $\Gamma \models \phi(t)$ for some $t \in \mathcal{L}$.
\end{theorem}
\begin{proof} Immediate from compactness and the weak existence property.
\end{proof}

\section{Comparison of $\textsf{BQL}_\textsf{CD}^r$ to $\textsf{BQL}_\textsf{CD}$} 

We can use the reduction theorem to show that $\textsf{BQL}_\textsf{CD}^r$ and $\textsf{BQL}_\textsf{CD}$ have exactly the same theorems. 

\begin{theorem} $\vdash_{[-1]} \phi$ iff $\vdash \phi$.
\end{theorem}
\begin{proof} The left-right direction is trivial. For the converse, suppose $\vdash \phi$. Then, by Reduction, $\vdash_{[-1]} \Box^n \phi$ for some $n$. Suppose for a reductio that $\not \vdash_{[-1]} \phi$. By an almost identical argument to the completeness theorem for $\vdash$, $\vdash_{[-1]}$ is complete with respect to the consequence relation over $\mathcal{L}^+$ obtained by dropping the restriction on $\models$ to reflexive worlds. Thus, there exists an $\mathcal{L}^+$-model $\mathfrak{M}$ such that $w \not \Vdash \phi$ for some $w \in \mathfrak{M}$. Add a chain of $n$ worlds below $w$ as follows:
$$
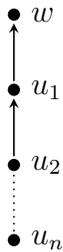
\begin{tikzpicture}[modal,node distance=1cm,world/.append style={minimum
size=1cm}]
\node[point] (root) [label=right: $w$] {};

\node[point] (w0) [below of=root, label=right: $u_1$] {};

\node[point] (w1) [below of=w0, label=right: $u_2$] {};

\node[point] (w2) [below of=w1, label=right: $u_n$] {};

\path[->] (w0) edge (root);

\path[->] (w1) edge (w0);

\path[-] (w2) edge[dotted] (w1);
\end{tikzpicture}
$$
Then $u_n \not \Vdash \Box^n \phi$. So, since $\vdash_{[-1]}$ is truth-preserving at every world in an $\mathcal{L}^+$-model, $\not \vdash_{[-1]} \Box^n \phi$, which is a contradiction.
\end{proof}

\section{Adding identity}

So far we have treated the identity predicate $=$ as a non-logical, theory-specific predicate. The question then arises as to whether the results obtained in this paper can be generalized to a setting in which $=$ is treated as a logical constant. Here we face a trade-off, which we also face when we try to add $=$ to intuitionistic logic. Model-theoretically, the most natural way to add $=$ to $\textsf{BQL}_\textsf{CD}^r$ is to restrict the class of $\mathcal{L}$-models to just those models $\mathfrak{M}$ such that for every world $w \in \mathfrak{M}$, $|=|(w) = \{\langle a, a \rangle: a \in dom(\mathfrak{M})\}$. However, this results in a logic which violates the disjunction property, since $t_1 = t_2 \vee (t_1 = t_2 \rightarrow \bot)$ is a theorem even though neither disjunct is. In order to preserve the disjunction property, we can instead restrict the class of $\mathcal{L}$-models to just those models such that $|=|(w)$ is a congruence relation at $w$ (i.e.\ $|=|(w)$ is an equivalence relation on the domain such that if $x_1$,...,$x_n$ are equivalent to $y_1$,...,$y_n$ respectively then (i) $|f^n|(x_1,...,x_n)$ is equivalent to $|f^n|(y_1,...,y_n)$ and (ii) $\langle x_1,...,x_n \rangle \in |R^n|(w)$ iff $\langle y_1,...,y_n \rangle \in |R^n|(w)$).\begin{footnote}{$|=|$ must still satisfy the persistence constraint.}\end{footnote} When we take the latter approach, completeness is obtained by adding the following rules to $\mathcal{N}\textsf{BQL}_\textsf{CD}^r$:
$$
[t = t]\hspace{2mm}(=\hspace{-1mm}\text{-Int})
\qquad
\infer[(=\hspace{-1mm}\text{-Elim})]
	{\phi(t_2)}
	{
	t_1 = t_2 
	&
	\phi(t_1)
	}
$$
All proofs and definitions (including the definition of the canonical model) are essentially unchanged. In particular, in the proof of Relative Deduction, $=$-Int is handled the same way as $\top$-Int and $=$-Elim falls under Case 2 of the inner induction steps.

Matters are \textit{slightly} more involved if we instead take the first approach and treat $=$ as ``real" identity. In this case, we obtain completeness by adding to $\mathcal{N}\textsf{BQL}_\textsf{CD}^r$ $=$-Int, $=$-Elim and excluded middle for identity (call the resulting logic $\textsf{BQL}_\textsf{CD=}^r$):
$$
[t_1 = t_2 \vee (t_1 = t_2 \rightarrow \bot)].
$$
We no longer prove completeness by defining a single canonical model, however. Rather, we define a different canonical model for each $\Gamma \not \vdash \phi$ ($|\Gamma| < \omega$). Since $\Gamma \not \vdash \phi$, we can find, by a similar argument to before, a prime saturated $\textsf{BQL}_\textsf{CD=}^r$-theory $\Gamma^* \supseteq \Gamma$ such that $\phi \not \in \Gamma^*$. The canonical model for $\Gamma \not \vdash \phi$ will now be a rooted model, with $\Gamma^*$ as the root and the worlds being just those prime saturated $\textsf{BQL}_\textsf{CD=}$-theories which $\Gamma^*$ has access to, where $\textsf{BQL}_\textsf{CD=}$ is the logic obtained by removing $\rightarrow$-Elim from the natural deduction system for $\textsf{BQL}_\textsf{CD=}^r$. Finally, we change the domain of the model from the set of closed terms to the set of equivalence classes of closed terms under the equivalence relation $\{\langle t_1, t_2 \rangle: t_1 = t_2 \in \Gamma^*\}$ and ``quotient-out" the intensions of relation symbols and extensions of constant symbols and function symbols in the standard way. Identity in this model is real because, by excluded middle for identity, $t_1 = t_2 \in \Sigma$ iff $t_1 = t_2 \in \Gamma^*$ for every prime saturated $\textsf{BQL}_\textsf{CD=}$-theory $\Sigma$ which $\Gamma^*$ has access to. 

\section{Bibliography}

[1] Bacon, A. (2013a). A New Conditional for Naive Truth Theory. \textit{Notre Dame Journal of Formal Logic}, 54(1), 87--104.

\noindent [2] Bacon, A. (2013b). Curry's Paradox and $\omega$-Inconsistency. \textit{Studia Logica}, 101, 1--9.

\noindent [3] Beall, J. (2009). \textit{Spandrels of Truth}. Oxford University Press.

\noindent [4] Brady, R. (2006). \textit{Universal Logic}. Center for the Study of Language and Information.

\noindent [5] Brady, R. (1984). Natural Deduction Systems for some Quantified Relevant Logics. \textit{Logique Et Analyse}, 27(8), 355--377.

\noindent [6] Field, H., Lederman H. \& \O gaard T. F. (2017). Prospects for a Naive Theory of Classes. \textit{Notre Dame Journal of Formal Logic}, 58(4), 461--506.

\noindent [7] Halbach, V. (2014). \textit{Axiomatic Theories of Truth}. Cambridge University Press.

\noindent [8] Kripke, S. (1975). Outline of a Theory of Truth. \textit{Journal of Philosophy}, 72(19), 690--716.

\noindent [9] Middleton, B. (2020). A Canonical Model for Constant Domain Basic First-Order Logic. \textit{Studia Logica}, 108, 1307--1323.

\noindent [10] Poizat, B. (2000). \textit{A Course in Model Theory: An Introduction to Contemporary Mathematical Logic}. Springer.  

\noindent [11] Priest, G. (2002). Paraconsistent Logic. In Gabbay D. M. \& Guenthner F. editors, \textit{Handbook of Philosophical Logic, 2nd Edition: Volume 6}, pp.\ 287--393. Kluwer Academic Publishers.

\noindent [12] Restall, G. (1994). Subintuitionistic Logics. \textit{Notre Dame Journal of Formal Logic}, 35(1), 116--129.

\noindent [13] Ruitenburg, W. (1998). Basic Predicate Calculus. \textit{Notre Dame Journal of Formal Logic}, 39(1), 18--46.

\noindent [14] Visser, A. (1981). A Propositional Logic with Explicit Fixed Points. \textit{Studia Logica}, 40, 155--175.

\appendix

\section{A standard model for \textsf{NT}}

Let $\mathcal{L}_\mathbb{N} = \mathcal{L}_T \setminus \{T\}$. An $\mathcal{L}_\mathbb{N}$-model $\mathfrak{M}$ is \textit{standard} iff (i) $dom(\mathfrak{M}) = \omega$, (ii) $|0| = 0$, (iii) $|f_e|$ is the primitive recursive function with index $e$ and (iv) for every $w \in \mathfrak{M}$: $|=|(w) = \{\langle n, n \rangle: n \in \omega\}$. A standard $\mathcal{L}_\mathbb{N}$-model therefore consists of copies of the classical standard model $\mathbb{N}$ connected by transitive arrows. For $\mathcal{L} \supseteq \mathcal{L}_\mathbb{N}$, an $\mathcal{L}$-model $\mathfrak{M}$ is standard iff the reduct of $\mathfrak{M}$ to $\mathcal{L}_\mathbb{N}$ is standard. We say that a set of $\mathcal{L}$-sentences $\Gamma$ \textit{has} a standard model iff there exists a standard $\mathcal{L}$-model $\mathfrak{M}$ such that $w \Vdash \Gamma$ for some reflexive $w \in \mathfrak{M}$. In this appendix, I build a standard $\mathcal{L}_T$-model $\mathbb{N}_T$ for \textsf{NT}. It follows that $\textsf{NT}$ is $\omega$-consistent in $\textsf{BQL}_\textsf{CD}^r$, in the sense that (i) $\textsf{NT} \vdash \phi(\dot{n})$ for every $n$ only if $\textsf{NT} \cup \forall v \phi \not \vdash \bot$ and (ii) $\textsf{NT} \vdash \exists v \phi$ only if $\textsf{NT} \cup \{\phi(\dot{n})\} \not \vdash \bot$ for some $n$. I build $\mathbb{N}_T$ using the positive Brady construction, a two-valued version of the Brady construction due to Field, Lederman and \O gaard (2017). The main difference between the construction as presented by Field-Lederman-\O gaard and the construction as presented here is that we add a loop at the end. But since Field-Lederman-\O gaard are working in the context of naive set theory, it is worth presenting the construction in full detail. The strategy is to start with a standard $\mathcal{L}_\mathbb{N}$-model consisting of a single dead-end copy of $\mathbb{N}$ and then, one by one, add irreflexive copies of $\mathbb{N}$ to a transfinitely descending chain below the dead-end copy. At each stage, we expand the model to $\mathcal{L}_T$ using a Kripke-like construction. As we descend down the chain, fewer conditionals are satisfied, which removes more counterexamples to modus ponens. Eventually, all counterexamples to modus ponens are removed. At this point, we can add a loop without disturbing the truth-value of any $\mathcal{L}_T$-sentence, which ensures $T$ keeps its intended extension. 

\subsection{The construction} Let $\{\mathfrak{M}_\alpha\}_{\alpha \in Ord}$ be a chain of standard $\mathcal{L}_\mathbb{N}$-models such that each $\mathfrak{M}_\alpha$ has the form:
$$
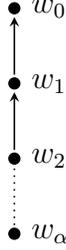
\begin{tikzpicture}[modal,node distance=1cm,world/.append style={minimum
size=1cm}]
\node[point] (root) [label=right: $w_0$] {};

\node[point] (w0) [below of=root, label=right: $w_1$] {};

\node[point] (w1) [below of=w0, label=right: $w_2$] {};

\node[point] (w2) [below of=w1, label=right: $w_{\alpha}$] {};

\path[->] (w0) edge (root);

\path[->] (w1) edge (w0);

\path[-] (w2) edge[dotted] (w1);
\end{tikzpicture}
$$
(arrows in diagrams are always transitive). We now define a chain $\{\mathfrak{M}_\alpha^T\}_{\alpha \in Ord}$ of $\mathcal{L}_T$-expansions of the $\mathfrak{M}_\alpha$ by induction on $\alpha$. Suppose $|T|(w_\beta)$ has already been defined for every $\beta < \alpha$.  For arbitrary $X \subseteq \omega$, let $\mathfrak{M}_\alpha^T[X]$ denote the object which \textit{would} be obtained were we to set $|T|(w_\alpha) = X$. $\mathfrak{M}_\alpha^T[X]$ is not necessarily an $\mathcal{L}_T$-model, since we need not have $X \subseteq |T|(w_\beta)$ for all $\beta < \alpha$. Nevertheless, we can still define satisfaction on $\mathfrak{M}^T_\alpha[X]$ in the same way as a real $\mathcal{L}_T$-model. Let $\Phi_\alpha(X) = \{\phi: \mathfrak{M}^T_\alpha[X], w_\alpha \Vdash \phi\}$. 

\begin{lemma}[Monotonicity] If $X \subseteq Y$ then $\Phi_\alpha(X) \subseteq \Phi_\alpha(Y)$.
\end{lemma}
\begin{proof} Suppose $X \subseteq Y$. We show by induction on the construction of $\mathcal{L}_T$-formulas that $\mathfrak{M}_\alpha^T[X], w_\alpha \Vdash \phi(\overline{n})$ only if $\mathfrak{M}_\alpha^T[Y], w_\alpha \Vdash \phi(\overline{n})$.

\underline{Base Cases} The claim holds trivially for atomic $\phi \neq T(t)$. For $\phi = T(t)$ we have
\begin{align*}
\mathfrak{M}_\alpha^T[X], w_\alpha \Vdash T(t)(\overline{n}) &\implies |t|(\overline{n}) \in X \\
&\implies |t|(\overline{n}) \in Y \\
&\implies \mathfrak{M}_\alpha^T[Y], w_\alpha \Vdash T(t)(\overline{n}).
\end{align*}
\underline{Induction Steps} The induction steps are standard except for $\rightarrow$, which holds due to the fact that $w_\alpha$ is irreflexive.
\end{proof}

\noindent We can now inductively define a sequence of increasingly better extensions for $T$ at $w_\alpha$ in the style of Kripke (1975):
\begin{align*}
X_\alpha(0) &= \emptyset \\
X_\alpha(\beta + 1) &= \Phi_\alpha(X_\alpha(\beta)) \\
X_\alpha(\gamma) &= \bigcup_{\beta < \gamma} X_\alpha(\beta) \hspace{5mm} \text{for } \gamma \text{ a limit}.
\end{align*}

\begin{lemma}[Locally Increasing] If $\beta \leq \beta'$ then $X_\alpha(\beta) \subseteq X_\alpha(\beta')$.
\end{lemma}
\begin{proof} By induction on $\beta$. The base case $\beta = 0$ holds trivially.

\underline{Successor Step} Suppose $\beta + 1 \leq \beta'$.  There are two cases.

\underline{Case 1} $\beta'$ is a successor. Then we have
\begin{align*}
\beta \leq \beta' - 1 &\implies X_\alpha(\beta) \subseteq X_\alpha(\beta' - 1) &&(\text{induction hypothesis}) \\
&\implies \Phi_\alpha(X_\alpha(\beta)) \subseteq \Phi_\alpha(X_\alpha(\beta' - 1)) &&(\text{Monotonicity}) \\
&\implies X_\alpha(\beta + 1) \subseteq X_\alpha(\beta').
\end{align*}

\underline{Case 2} $\beta'$ is a limit. Then, trivially, $X_\alpha(\beta + 1) \subseteq X_\alpha(\beta')$.

\underline{Limit Step} Suppose $\beta \leq \beta'$ for $\beta$ a limit. Suppose $n \in X_\alpha(\beta)$.  Then $n \in X_\alpha(\beta_0)$ for some $\beta_0 < \beta$.  By the induction hypothesis, $X_\alpha(\beta_0) \subseteq X_\alpha(\beta')$.  So $n \in X_\alpha(\beta')$. 
\end{proof}

\begin{lemma}[Locally Convergent] There exists $\beta$ such that $X_\alpha(\beta) = X_\alpha(\beta')$ for all $\beta' \geq \beta$.
\end{lemma}
\begin{proof} Suppose not.  Then, by Locally Increasing, for every $\beta$ there exists $\beta' > \beta$ such that $X_\alpha(\beta) \subset X_\alpha(\beta')$, which contradicts the fact that $\bigcup_{\beta \in Ord} X_\alpha(\beta)$ is a set.
\end{proof}

\noindent We now define $|T|(w_\alpha) = X_\alpha(\alpha^+)$, where $\alpha^+$ is the least $\beta$ such that $X_\alpha(\beta) = X_\alpha(\beta')$ for all $\beta' \geq \beta$.  This completes the definition of $\mathfrak{M}_\alpha^T$. 

\begin{lemma}[Closure] $w_\alpha \Vdash \phi$ iff $\phi \in |T|(w_\alpha)$.
\end{lemma}
\begin{proof} We have
\begin{align*}
w_\alpha \Vdash \phi &\iff \phi \in \Phi_\alpha(|T|(w_\alpha)) \\
&\iff \phi \in \Phi_\alpha(X_\alpha(\alpha^+)) \\
&\iff \phi \in X_\alpha(\alpha^+ + 1) \\
&\iff \phi \in X_\alpha(\alpha^+) \\
&\iff \phi \in |T|(w_\alpha).
\end{align*}
\end{proof}

\begin{lemma}[Globally Decreasing] If $\alpha \leq \beta$ then $|T|(w_\beta) \subseteq |T|(w_\alpha)$.
\end{lemma}
\begin{proof} Suppose $\alpha \leq \beta$.  

\begin{subclaim} For all $\xi: X_\beta(\xi) \subseteq X_\alpha(\xi)$. 
\end{subclaim}
\begin{subproof}
By induction on $\xi$. The base case $\xi = 0$ holds trivially.

\underline{Successor Step} Suppose $X_\beta(\xi) \subseteq X_\alpha(\xi)$. Since $\alpha \leq \beta$, every conditional satisfied at $w_\beta$ is also satisfied at $w_\alpha$. So, by a similar argument to Monotonicity: $\mathfrak{M}_\beta^T[X_\beta(\xi)], w_\beta \Vdash \phi(\overline{n})$ only if $\mathfrak{M}_\alpha^T[X_\alpha(\xi)], w_\alpha \Vdash \phi(\overline{n})$. Therefore $X_\beta(\xi + 1) \subseteq X_\alpha(\xi + 1)$.

\underline{Limit Step} Let $n \in X_\beta(\gamma)$ for $\gamma$ a limit.  Then $n \in X_\beta(\xi)$ for some $\xi < \gamma$.  By the induction hypothesis, $X_\beta(\xi) \subseteq X_\alpha(\xi)$.  So $n \in X_\alpha(\xi) \subseteq X_\alpha(\gamma)$.  
\end{subproof}

\noindent There are now two cases to consider.

\underline{Case 1} $\beta^+ = \alpha^+ + \xi$ for some ordinal $\xi$.  Then 
\begin{align*}
|T|(w_\beta) &= X_\beta(\beta^+) \\
&= X_\beta(\alpha^+ + \xi) \\
&\subseteq X_\alpha(\alpha^+ + \xi) &&(\text{Subclaim 1}) \\
&= X_\alpha(\alpha^+) \\
&= |T|(w_\alpha).
\end{align*}

\underline{Case 2} $\alpha^+ = \beta^+ + \xi$ for some ordinal $\xi$. Then
\begin{align*}
|T|(w_\beta) &= X_\beta(\beta^+) \\
&= X_\beta(\beta^+ + \xi) \\
&= X_\beta(\alpha^+) \\
&\subseteq X_\alpha(\alpha^+) &&(\text{Subclaim 1}) \\
&= |T|(w_\alpha).
\end{align*}
\end{proof}

\noindent It follows from Globally Decreasing that $\mathfrak{M}_\alpha^T$ is in fact an $\mathcal{L}_T$-model. Accordingly, $\mathfrak{M}_\alpha^T$ satisfies Persistence, which allows us to prove that we eventually reach an ordinal $\alpha$ such that for all $\beta \geq \alpha$: $w_\beta \Vdash \phi(\overline{n})$ iff $w_\alpha \Vdash \phi(\overline{n})$. Let $S(\alpha) = \{\langle \phi(\overline{v}), \langle \overline{v} \rangle, \langle \overline{n} \rangle \rangle: w_\alpha \Vdash \phi(\overline{n})\}$.

\begin{lemma}[Globally Convergent]  There exists $\alpha$ such that for all $\beta \geq \alpha: S(\beta) = S(\alpha)$.
\end{lemma}
\begin{proof} Suppose not.  Then, by Persistence, for every $\alpha$ there exists $\beta > \alpha$ such that $S(\beta) \subset S(\alpha)$, which contradicts the fact that $S(0)$ is a set.
\end{proof}

\noindent Let $\Theta$ denote the least $\alpha$ such that $S(\alpha) = S(\beta)$ for all $\beta \geq \alpha$. We define $\mathbb{N}_T$ to be the $\mathcal{L}_T$-model obtained from $\mathfrak{M}_\Theta^T$ by adding a loop at $\Theta$:
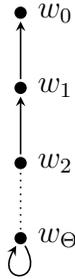
$$
\begin{tikzpicture}[modal,node distance=1cm,world/.append style={minimum
size=1cm}]
\node[point] (root) [label=right: $w_0$] {};

\node[point] (w0) [below of=root, label=right: $w_1$] {};

\node[point] (w1) [below of=w0, label=right: $w_2$] {};

\node[point] (w2) [below of=w1, label=right: $w_\Theta$] {};

\path[->] (w0) edge (root);

\path[->] (w1) edge (w0);

\path[-] (w2) edge[dotted] (w1);

\path[->] (w2) edge[reflexive below] (w2);
\end{tikzpicture}
$$

\noindent Note that for all $n < \omega$, $w_n \Vdash \Box^{n + 1} \bot$ but $w_{n + 1} \not \Vdash \Box^{n + 1} \bot$. Thus, $\Theta \geq \omega$. In fact, if we define $\Box^{<\omega} \phi = \exists xT$\d{$\Box$}$^x\ulcorner \phi \urcorner$ then we have $w_{\omega + n} \Vdash \Box^{n + 1}\Box^{<\omega}\bot$ but $w_{\omega + n + 1} \not \Vdash \Box^{n + 1}\Box^{<\omega} \bot$.\begin{footnote}{We follow the notational conventions of Halbach (2014), where the result of placing a dot beneath a sentential operator $O$ abbreviates the function symbol for the p.r.\ function corresponding to $O$.}\end{footnote}  Next we can define $\Box^{<\omega 2} \phi = \Box^{<\omega} \Box^{<\omega} \phi$ and so on up to $\omega^2$, where we can define $\Box^{<\omega^2}\phi = \exists x T($\d{$\Box^{<\omega}$}$)^x \ulcorner \phi \urcorner$ and keep going. So $\Theta$ will be located a decent way out into the ordinals, although $\Theta < \omega_1$ since $S(0)$ is countable.\begin{footnote}{$\Theta$ can equivalently be characterized as the least $\alpha$ such that $S(\alpha) = S(\alpha + 1)$.}\end{footnote} 

\begin{lemma}[Standard Model] $\mathbb{N}_T, w_\Theta \Vdash \textsf{NT}$.
\end{lemma}
\begin{proof} The only non-trivial axioms are the Tarski biconditionals.\begin{footnote}{Note, however, that although $(\mathbb{N}_T, w_\Theta)$ validates the induction schema in \textsf{NT}, $(\mathbb{N}_T, w_\Theta)$ does not validate the more usual formulation $\forall \overline{x}[\phi(0, \overline{x}) \wedge \forall y(\phi(y, \overline{x}) \rightarrow \phi(s(y), \overline{x})) \rightarrow \forall y\phi(y, \overline{x})]$.}\end{footnote} We need to verify that $\phi \in |T|(w_\Theta)$ iff $\mathbb{N}_T, w_\Theta \Vdash \phi$. By Closure, it suffices to show that $\mathfrak{M}_\Theta^T, w_\Theta \Vdash \phi(\overline{n})$ iff $\mathbb{N}_T, w_\Theta \Vdash \phi(\overline{n})$, which we prove by induction on the construction of $\mathcal{L}_T$-formulas. The base case is easy. The induction steps are also easy except for $\rightarrow$.

\underline{$\rightarrow$} $\impliedby$ This direction is trivial.

$\implies$ Suppose $\mathbb{N}_T, w_\Theta \not \Vdash (\phi \rightarrow \psi)(\overline{n})$.  Then $\mathbb{N}_T, w_\alpha \Vdash \phi(\overline{n})$ and $\mathbb{N}_T, w_\alpha \not \Vdash \psi(\overline{n})$ for some $\alpha \leq \Theta$.  If $\alpha < \Theta$ then we're done. Suppose $\alpha = \Theta$.  Then, by the induction hypothesis: $\mathfrak{M}_\Theta^T, w_\Theta \Vdash \phi(\overline{n})$ and $\mathfrak{M}_\Theta^T, w_\Theta \not \Vdash \psi(\overline{n})$.  So $\mathfrak{M}_{\Theta + 1}^T, w_{\Theta + 1} \not \Vdash (\phi \rightarrow \psi)(\overline{n})$. But then $\mathfrak{M}_\Theta^T, w_\Theta \not \Vdash (\phi \rightarrow \psi)(\overline{n})$.
\end{proof}

\noindent Note that $(\mathbb{N}_T, w_\Theta)$ does not validate all sentences which are true on the classical standard model of arithmetic. For example, due to the fact that $w_0$ is a dead-end, $\mathbb{N}_T, w_\Theta \not \Vdash (0 = 0 \rightarrow \bot) \rightarrow \bot$.

\begin{theorem}[$\omega$-Consistency] (1) If $\textsf{NT} \vdash \phi(\dot{n})$ for all $n$ then $\textsf{NT} \cup \{\forall v \phi\} \not \vdash \bot$, (2) if $\textsf{NT} \vdash \exists v \phi$ then $\textsf{NT} \cup \{\phi(\dot{n})\} \not \vdash \bot$ for some $n$.
\end{theorem}
\begin{proof} Immediate from soundness and Standard Model.
\end{proof}

\begin{corollary}[Non-Triviality] $\textsf{NT} \not \vdash \bot$.
\end{corollary}

\noindent It also follows from the existence of the standard model that (i) the theory obtained by simultaneously closing $\textsf{NT}$ under both $\textsf{BQL}_\textsf{CD}^r$ and the $\omega$-rule is non-trivial and (ii) if $\textsf{NT} \vdash \phi(\dot{n})$ for all $n$ then $\textsf{NT} \not \vdash \exists v (\phi \rightarrow \bot)$ (the latter property is called ``strong $\omega$-consistency" in Bacon (2013b)).
\end{document}